%% file: quircuits_extended160525.tex
\documentclass{article}
\include{minimalJ}

\newcommand{\Rcal}{{\cal R}}
\usepackage{verbatim}

\title{Reconstruction of infinite matroids from their 3-connected minors}
\author{Nathan Bowler \and Johannes Carmesin \and Luke Postle}
\begin{document}
\maketitle
\begin{abstract}
We show that any infinite matroid can be reconstructed from the torsos of a tree-decomposition over its 2-separations, together with local information at the ends of the tree. We show that if the matroid is tame then this local information is simply a choice of whether circuits are permitted to use that end. The same is true if each torso is planar, with all gluing elements on a common face.
\end{abstract}

\section{Introduction}

In \cite{matroid_axioms}, Bruhn, Diestel, Kriesell, Pendavingh and Wollan introduced axioms for infinite matroids 
in terms of independent sets, bases, circuits, closure and relative rank. 
These axioms allow for duality of infinite matroids as known from finite matroid theory, which settles an old problem of Rado. This breakthrough allowed the development of the basic theory of infinite matroids (see for example \cite{RD:HB:graphmatroids, HB:PW:connectivity, NB:JC:PC}). 

In \cite{ADP:decomposition}, Aigner-Horev, Diestel and Postle showed that any (infinite) matroid has a canonical tree-decomposition over its 2-separations. More precisely, a {\em tree-decomposition of adhesion 2} of a matroid $N$ consists of a tree $T$ and a partition $R = (R_t)_{t \in V(T)}$ of the ground set $E$ of $N$ such that for any element $tt'$ of $T$ the partition of the ground set induced by that element is a 2-separation of $N$. At each node $t$ of $T$ this gives a {\em torso}: a matroid on a set consisting of $R_t$ together with some new {\em virtual elements} corresponding to the edges incident with $t$ in $T$. What Aigner-Horev, Diestel and Postle showed is that any connected matroid has a canonical tree-decomposition of adhesion 2 such that all torsos are either 3-connected or else are circuits or cocircuits.

If the matroid $N$ is finite then it can easily be reconstructed from this decomposition, by taking the 2-sum of all the torsos. This is no longer possible for matroids in general. Consider for example the graph $Q$ obtained by taking the 2-sum of a ray of copies of $K_4$, as in \autoref{fig:quircuits2}. The finite-cycle matroid $M_{FC}(Q)$ of $Q$ has as its circuits the edge sets of finite cycles in $Q$. The topological-cycle matroid $M_{TC}(Q)$ has as its circuits the edge sets of finite cycles and double rays in $Q$ (see \cite{RD:HB:graphmatroids} for the definition of topological-cycle matroids of general graphs). Both of these matroids have the same tree-decomposition into torsos, consisting of a ray along which all torsos are the cycle matroid of the graph $K_4$, as in \autoref{fig:quircuits2}.

\begin{figure}[htb]
\begin{center}
 \includegraphics[width=8cm]{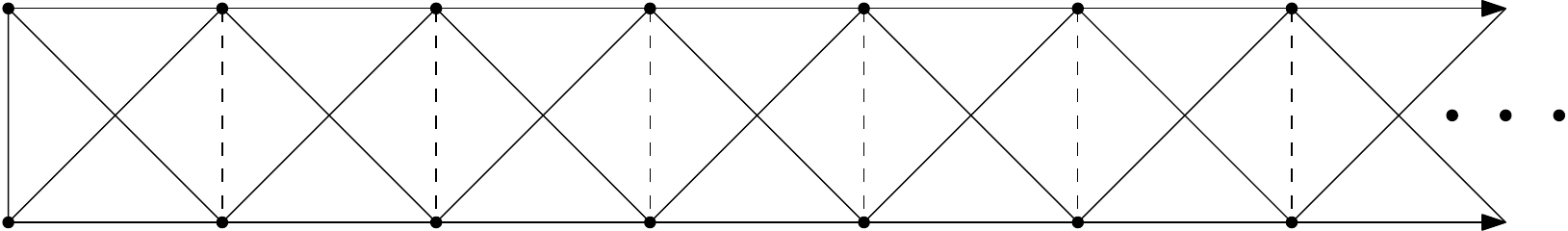}
\end{center} 
\caption{The graph $Q$ as constructed from a ray of copies of $K_4$: solid edges are edges of $Q$, dotted edges are eliminated when we take 2-sums} \label{fig:quircuits2}
\end{figure}

Thus in order to reconstruct these two matroids we would need some additional information distinguishing them: in this case, the choice of whether to allow circuits which go infinitely far along the ray. We shall show that if the matroid we have decomposed is {\em tame} (that is, any intersection of a circuit with a cocircuit is finite) then such a choice, for each end of the tree, is all the extra information that we need.

\begin{thm}\label{tamemain}
Let $R = (R_t)_{t \in V(T)}$ be a tree-decomposition of adhesion 2 of a tame matroid $N$. Let $\Psi$ be the set of ends $\omega$ of $T$ which are {\em used by} circuits of $N$, in the sense that there is a circuit $o$ of $N$ for which $\omega$ lies in the closure of $\{t \in V(T) | o \cap R_t \neq \emptyset\}$. Then $N$ is uniquely determined by the tree $T$, the torsos $(N_t | t \in T)$ and the set $\Psi$.
\end{thm}

Note that this theorem applies to arbitrary adhesion 2 tree-decompositions of $N$, not just to the canonical one mentioned above. It applies only to tame matroids, but this is known to be a rich and very natural class of matroids, as explained in \cite{BC:wildmatroids, BC:psi_matroids}.

The process by which the matroid $N$ can be rebuilt is a generalisation of 2-sums to infinite trees of matroids developed in \cite{BC:psi_matroids}. This construction also provides a tool for building new matroids. In \cite{BC:psi_matroids} we characterised, for a given tree of matroids as above, for which sets $\Psi$ of ends of the tree we get a matroid by this infinitary 2-sum operation: in particular, this is true whenever $\Psi$ is Borel. 

What if the matroid to be reconstructed is not tame? In this case, the information given by the tree of torsos and the set of ends which are used is woefully inadequate. Consider once more the graph $Q$ from \autoref{fig:quircuits2}. We say that two edge sets of double rays of $Q$ are {\em equivalent} if they have finite symmetric difference (this means that the double rays have the same tails). Let $\Ccal$ be the union of any collection of equivalence classes with respect to this relation, together with the class of all edge sets of finite cycles in $Q$. We will show in \autoref{sec:ray_case} that any such $\Ccal$ is the set of circuits of a matroid. Furthermore, all these matroids have the same canonical tree-decomposition and the same torsos as $M_{FC}(Q)$ and $M_{TC}(Q)$. Accordingly, we need a great deal of new information to reconstruct the matroid.

Even worse, we may need such information to be provided for each of uncountably many ends. In \autoref{example_binary_tree}, we will construct a graph $W$ (pictured in \autoref{fig:Wint}) by sticking together many copies of $Q$ along the rays of a binary tree. We will do this in such a way that any topological circuit of $W$ which uses an end $\omega$ looks, in the vicinity of $\omega$, like a double ray in the copy of $Q$ running to $\omega$. If we choose a set of equivalence classes as above for each of these copies of $Q$ then we get a matroid whose circuits are the circuits of the topological cycle matroid of $W$ which match one of the chosen equivalence classes at each end they use. All of these matroids have the same tree-decomposition into the same torsos. This example is discussed in more detail in \autoref{example_binary_tree}. Nevertheless, all the extra information used here is local to the ends, in the sense that the choices made at the end $\omega$ only affect what circuits can do close to $\omega$.
\begin{figure}[htb]
\begin{center}
 \includegraphics[width=8cm]{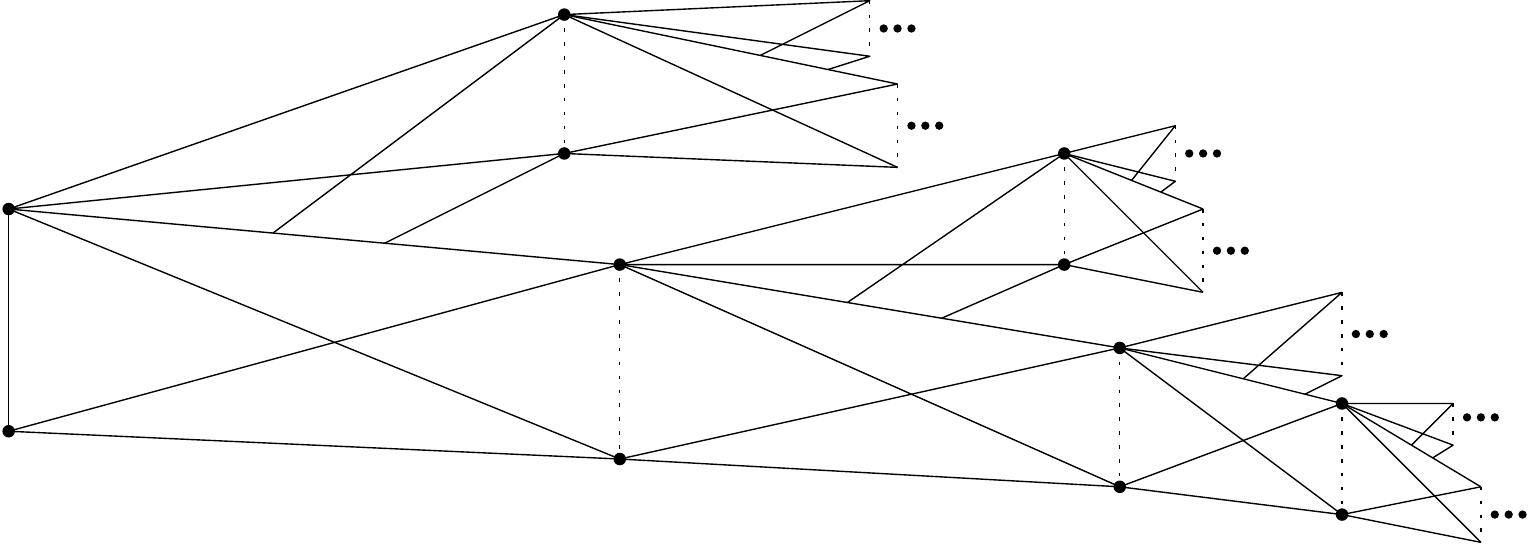}
\end{center} 
\caption{The graph $W$.} \label{fig:Wint}
\end{figure}

We show that this is true in general.

\begin{thm}\label{main:thm_intro}
Let $R = (R_t)_{t \in V(T)}$ be a tree-decomposition of adhesion 2 of a (not necessarily tame) matroid $N$. Then $N$ is uniquely determined by the tree $T$, the torsos $(N_t | t \in T)$ and local information at the ends of $T$.
\end{thm}

The precise form of this local information is explained in \autoref{sec:main_result}. From this more general result we will deduce \autoref{tamemain}, which says that  for tame matroids the only local information we need is which ends are used.

There is another case in which the local information collapses to simply the information about which ends are used. We say that a tree-decomposition of a matroid is {\em planar} if each of the torsos can be represented as the finite cycle matroid of some finite graph embedded in the plane, in such a way that all of the virtual elements lie on the boundary of the outer face. 

\begin{cor}\label{cor:planar}
Let $R = (R_t)_{t \in V(T)}$ be a planar tree-decomposition of adhesion 2 of a matroid $N$. Then $N$ is uniquely determined by the tree $T$, the torsos $(N_t | t \in T)$ and the set of ends of $T$ used by circuits of $N$. 
\end{cor}

The wild matroids constructed from the graph $Q$ are of interest in their own right as counterexamples. In particular, there is a natural question if the characterisations of finite binary matroids are equivalent in the infinite setting. For the wild matroids arising from $Q$, these equivalences break dow; these matroids satisfy roughly half of the basic characterisations of finite binary matroids (see \cite{Oxley} for many such characterisations). 
For example, none of these matroids has a $U_{2,4}$-minor, and all of them have the property that any symmetric difference of two circuits is a disjoint union of circuits. But in one of them there is a set of three circuits whose symmetric difference is not a disjoint union of circuits. See \autoref{repQ} for details.

The paper is organised as follows: In \autoref{sec:prelims} we recall some basic facts and 
give a new simple proof that the torso of a canonical decomposition is a matroid. 
Before proving our main result in \autoref{sec:main_result}, we first deal with the subcase where the torsos are glued together along a ray, see \autoref{sec:ray_case}.
Finally, in \autoref{example_binary_tree} we explain the example indicated in the Introduction before the statement of \autoref{main:thm_intro}.

\section{Preliminaries}\label{sec:prelims}
\subsection{Infinite matroids}

We begin by recalling two of the axiomatisations of infinite matroids given in \cite{matroid_axioms}.

A set system $\Ical\subseteq \Pcal(E)$ is the set of independent sets of a matroid iff it satisfies the following {\em independence  axioms\/}.
\begin{itemize}
	\item[(I1)] $\varnothing\in \Ical(M)$.
	\item[(I2)] $\Ical(M)$ is closed under taking subsets.
	\item[(I3)] Whenever $I,I'\in \Ical(M)$ with $I'$ maximal and $I$ not maximal, there exists an $x\in I'\setminus I$ such that $I+x\in \Ical(M)$.
	\item[(IM)] Whenever $I\subseteq X\subseteq E$ and $I\in\Ical(M)$, the set $\{I'\in\Ical(M)\mid I\subseteq I'\subseteq X\}$ has a maximal element.
\end{itemize}

The {\em bases} of such a matroid are the maximal independent sets, and the {\em circuits} are the minimal dependent sets.

A set system $\Ccal\subseteq \Pcal(E)$ is the set of circuits of a matroid iff it satisfies the following {\em circuit  axioms\/}.
\begin{itemize}
\item[(C1)] $\varnothing\notin\Ccal$.
\item[(C2)] No element of $\Ccal$ is a subset of another.
\item[ (C3)](Circuit elimination) Whenever $X\subseteq o\in \Ccal(M)$ and $\{o_x\mid x\in X\} \subseteq \Ccal(M)$ satisfies $x\in o_y\Leftrightarrow x=y$ for all $x,y\in X$, 
then for every $z \in o\setminus \left( \bigcup_{x \in X} o_x\right)$ there exists a  $o'\in \Ccal(M)$ such that $z\in o'\subseteq \left(o\cup  \bigcup_{x \in X} o_x\right) \setminus X$.

\item[(CM)] $\Ical$ satisfies (IM), where $\Ical$ is the set of those subsets of $E$ not including an element of $\Ccal$.
\end{itemize}

The following basic lemmas may be proved as for finite matroids, see for example \cite{BC:psi_matroids}.

\begin{lem}\label{fdt}
 Let $M$ be a matroid and $s$ be a base.
Let $o_e$ be a fundamental circuit with respect to $s$ and $b_f$ be a fundamental cocircuit with respect to $s$. Then:
\begin{enumerate}
 \item $o_e\cap b_f$ is empty or $o_e\cap b_f=\{e,f\}$ and
\item $f\in o_e$ if and only if $e\in b_f$.
\end{enumerate}
\end{lem}

\begin{lem}\label{o_cap_b}
 For any circuit $o$ with at least two elements $e$ and $f$, there is a cocircuit $b$ such that $o\cap b=\{e,f\}$.
\end{lem}

\begin{lem} \label{rest_cir}
 Let $M$ be a matroid with ground set $E = C \dot \cup X \dot \cup D$ and let $o'$ be a circuit of $M' = M / C \backslash D$.
Then there is an $M$-circuit $o$ with $o' \subseteq o \subseteq o' \cup C$.
\end{lem}

A \emph{scrawl} is a union of circuits. In \cite{BC:rep_matroids},
(infinite) matroids are axiomatised in terms of scrawls. 
The set $\Scal(M)$ denotes the set of scrawls of the matroid $M$.
Dually a \emph{coscrawl} is a union of cocircuits.
Since no circuit and cocircuit can meet in only one element,
no scrawl and coscrawl can meet in only one element. 
In fact, this property gives us a simple characterisation of scrawls in terms of coscrawls and vice versa.

\begin{lem}\label{is_scrawl}
Let $M$ be a matroid, and let $w\subseteq E$. The following are equivalent:
\begin{enumerate}
 \item $w$ is a scrawl of $M$.
 \item $w$ never meets a cocircuit of $M$ just once.
 \item $w$ never meets a coscrawl of $M$ just once.
\end{enumerate}
\end{lem}

The dual of the statement proved in the following Lemma appears as axiom (O3$^*$) in an axiomatisation discussed in \cite{BC:psi_matroids}. We will need that this statement holds in any matroid.

\begin{lem}
 Let $M$ be a matroid, and let $o$ be a circuit of $M$, $X$ a subset of the ground set $E$ of $M$ and $e \in o \setminus X$. Then there is a circuit $o_{min}$ of $M$ with $e \in o_{min} \subseteq X \cup o$ and such that $o_{min} \sm X$ is minimised subject to these conditions.
\end{lem}
\begin{proof}
$o \setminus X$ never meets a cocircuit $b$ of $M / X$ just once: if it did, then the circuit $o$ of $M$ would meet the cocircuit $b$ of $M$ just once. So by \autoref{is_scrawl}, $o \setminus X$ is a union of circuits of $M/X$. Let $o'$ be a circuit of $M/X$ with $e \in o' \subseteq o \setminus X$, and take $o_{min}$ to be a circuit of $M$ with $o' \subseteq o_{min} \subseteq o' \cup X$ (which exists by \autoref{rest_cir}). Then for any other circuit $\tilde o$ of $M$ with $e \in \tilde o \setminus X \subseteq o_{min} \setminus X$ we have that $\tilde o \setminus X$ is a nonempty union of circuits of $M/X$ included in the circuit $o'$, so is equal to $o'$, so that $o_{min} \setminus X \subseteq \tilde o \setminus X$. This establishes the minimality of $o_{min} \setminus X$.
\end{proof}

\begin{lem}\label{cir_c_scrawl}
 Let $M$ be a matroid and $\Ccal,\Dcal\se \Pcal(E)$
such that every $M$-circuit is a union of elements of $\Ccal$, every $M$-cocircuit is a union of elements of $\Dcal$ and $|C\cap D|\neq 1$ for every $C\in\Ccal$ and every $D\in\Dcal$.

Then $\Ccal(M)\se \Ccal\se \Scal(M)$ and  $\Ccal(M^*)\se \Dcal\se \Scal(M^*)$.
\end{lem}

\begin{proof}
We begin by showing that $\Ccal(M) \se \Ccal$. Let $o$ be a circuit of $M$ and $e$ be an element of $o$. Since $o$ is a union of elements of $\Ccal$ there is some $o' \in \Ccal$ with $e \in o' \se o$. Suppose for a contradiction that $o'$ is a proper subset of $o$, so that there is some $f \in o \sm o'$. By \autoref{o_cap_b} there is a cocircuit $b$ of $M$ with $o' \cap b = \{e\}$. But then there is some $b' \in \Dcal$ with $e \in b' \se b$ such that $o' \cap b' = \{e\}$, giving the desired contradiction. Dually we obtain that $\Ccal(M^*) \se \Dcal$.

The fact that $\Ccal \se \Scal(M)$ is immediate from \autoref{is_scrawl} since $\Ccal(M^*) \se \Dcal$, and the proof that $\Dcal \se \Scal(M^*)$ is similar.
\end{proof}
We will need a lemma showing that any circuit can be obtained from any other by a single application of (C3).

\begin{lem}\label{good_cir_eli}
Let $o$ and $o'$ be circuits of a matroid $M$, and let $z$ be in both $o$ and $o'$. Then there are a set $X \subseteq o - z$ and a family $(o_x | x \in X)$ of circuits of $M$ with $o_x \cap (X \cup \{z\}) = \{x\}$ for each $x \in X$ such that the only circuit $o''$ with $z \in o'' \subseteq o \cup (\bigcup_{x \in X} o_x) \setminus X$ is $o'$.
\end{lem}
\begin{proof}
Let $M' = M/(o' - z)$, and let $s$ be a base of $o \setminus o'$ in $M'$. We take $X = (o \setminus o') \setminus s$, and for each $x \in X$ we take $\hat o_x$ to be the fundamental circuit of $x$ with respect to $s$ and $o_x$ to be an $M$-circuit with $\hat o_x \subseteq o_x \subseteq \hat o_x \cup (o' - z)$, which exists by Lemma \ref{rest_cir}. Thus $o_x \cap (X \cup \{z\}) = \{x\}$ for each $x \in X$, and $z \in o' \subseteq o \cup (\bigcup_{x \in X} o_x) \setminus X$. Now suppose for a contradiction that there is some other circuit $o''$ such that $z \in o'' \subseteq o \cup (\bigcup_{x \in X} o_x) \setminus X$, and let $e \in o'' \setminus o'$. Then also $e \in o \setminus o' \setminus X = s$. Since $s$ is coindependent in $M'$, there must be some cocircuit $b$ of $M'$ with $b \cap s = \{e\}$. Since $M'$ is a contraction of $M$, the cocircuit $b$ is also a cocircuit of $M$, and so it cannot meet $o''$ only in $e$, and hence it must contain some point of $o'$. But the only point of $o'$ that $b$ can contain is $z$, so 
$b \cap o' = \{z\}$, which is the desired contradiction.
\end{proof}

\subsection{A hybrid axiomatisation}

In~\cite{BC:psi_matroids}, the following axioms were used as part of an axiomatisation of countable matroids:
\begin{itemize}
\item[(O1)] $|C\cap D|\neq 1$ for all $C\in \Ccal$ and $D\in \Dcal$. 
\item[(O2)] For all partitions $E=P\dot\cup Q\dot\cup \{e\}$
either $P+e$ includes an element of $\Ccal$ through $e$ or
$Q+e$ includes an element of $\Dcal$ through $e$.
\end{itemize}

In this paper, we will not restrict our attention to countable matroids, so we cannot use this axiomatisation. However, these two axioms will still be useful to us, because of the following lemma.

For a set $\Ccal\subseteq \Pcal(E)$, let $\Ccal^\perp$
be the set of those subsets of $E$ that meet no element of $\Ccal$ just once.
Note that $(O1)$ is equivalent to $\Dcal$ being a subset of $\Ccal^\perp$.

\begin{lem}[\cite{BC:psi_matroids}]\label{cireli_partitioning}
Let $\Ccal \subseteq \Pcal(E)$.
Then $\Ccal$ and $\Ccal^\perp$ satisfy $(O2)$ if and only if
$\Ccal$ satisfies circuit elimination $(C3)$. 
\end{lem}

This means that if we know what the cocircuits of a matroid ought to be then we can use (O1) and (O2) as a more symmetric substitute for the circuit elimination axiom (C3). More precisely:

\begin{thm}\label{hybrid}
Let $\Ccal$ and $\Dcal$ be sets of subsets of a set $E$. Then there is a matroid whose set of circuits is $\Ccal$ and whose set of cocircuits is $\Dcal$ if and only if all of the following conditions hold:
\begin{enumerate}
\item Each of $\Ccal$ and $\Dcal$ satisfies (C1) and (C2).
\item $\Ccal$ and $\Dcal$ satisfy (O1) and (O2).
\item $\Ccal$ satisfies (CM)
\end{enumerate}
\end{thm}

\begin{proof}
It is clear that all of these conditions are satisfied by the sets of circuits and cocircuits of any matroid. So it suffices to prove the reverse implications. Suppose that all 3 conditions hold. Then $\Ccal$ satisfies (C1), (C2) and (CM) by assumption, and satisfies (C3) by Lemma \ref{cireli_partitioning} and the fact that $\Dcal \subseteq \Ccal^\perp$. So there is a matroid $M$ whose set of circuits is $\Ccal$. Then by (O1) and Lemma \ref{is_scrawl}, every element $b$ of $\Dcal$ is a coscrawl of $M$ and so, being nonempty by (C1), includes some cocircuit $b'$ of $M$. Let $e \in b'$. On the other hand, for any cocircuit $b'$ of $M$ there is no circuit of $M$ meeting $b'$ exactly once, so by (O2) applied to the partition $E = (E \setminus b') \dot \cup (b' - e) \dot \cup \{e\}$ there must be some $b'' \in \Dcal$ with $e \in b'' \subseteq b'$. Then by (C2) we have $b'' = b$ and hence $b = b'$, thus $b$ is a cocircuit of $\Dcal$. We have now shown that every cocircuit of $M$ includes an element of $\Dcal$, and it 
follows that it must be that element. So $\Dcal$ is the set of cocircuits of $M$, as required.
\end{proof}

\subsection{Trees of matroids}\label{treesofmatroids}

In this section we review the relationship between trees of matroids and tree-decompositions of matroids.

\begin{dfn}
A {\em tree $\Tcal$ of matroids} consists of a tree $T$, together with a function $M$ assigning to each node $t$ of $T$ a matroid $M(t)$ on ground set $E(t)$, such that for any two nodes $t$ and $t'$ of $T$, if $E(t) \cap E(t')$ is nonempty then $tt'$ is an edge of $T$.

For any edge $tt'$ of $T$ we set $E(tt') = E(t) \cap E(t')$. We also define the {\em ground set} of $\Tcal$ to be $E = E(\Tcal) = \left(\bigcup_{t \in V(T)} E(t)\right) \setminus \left(\bigcup_{tt' \in E(T)} E(tt')\right)$. 

We shall refer to the elements which appear in some $E(t)$ but not in $E$ as {\em virtual elements} of $M(t)$; thus the set of such virtual elements is $\bigcup_{tt' \in E(T)} E(tt')$.
\end{dfn}

The idea is that the virtual elements are to be used only to give information about how the matroids are to be pasted together, but they will not be present in the final pasted matroid, which will have ground set $E(\Tcal)$. 
%Since all the $M(t)$ are finite, the tree $T$ must be locally finite.

\begin{dfn}
A tree $\Tcal = (T, M)$ of matroids is {\em of overlap 1} if, for every edge $tt'$ of $T$, $|E(tt')| = 1$. In this case, we denote the unique element of $E(tt')$ by $e(tt')$.

Given a tree of matroids of overlap 1 as above, a {\em precircuit} $(C, o)$ of $\Tcal$ consists of a connected subtree $C$ of $T$ together with a function $o$ assigning to each vertex $t$ of $C$ a circuit of $M(t)$, such that for any vertex $t$ of $C$ and any vertex $t'$ adjacent to $t$ in $T$, $e(tt') \in o(t)$ if and only if $t' \in C$. Given a set $\Psi$ of ends of $T$, such a precircuit is called a {\em $\Psi$-precircuit} if all ends of $C$ are in $\Psi$. The set of $\Psi$-precircuits is denoted $\overline\Ccal(\Tcal, \Psi)$. 

Any $\Psi$-precircuit $(C, o)$ has an {\em underlying set} $\underline{(C, o)} = E \cap \bigcup_{t \in V(C)} o(t)$. Minimal nonempty subsets of $E$ arising in this way are called {\em $\Psi$-circuits} of $\Tcal$. The set of $\Psi$-circuits of $\Tcal$ is denoted $\Ccal(\Tcal, \Psi)$.
\end{dfn}

\begin{dfn}
Let $\Tcal = (T, M)$ be a tree of matroids. Then the {\em dual} $\Tcal^*$ of $\Tcal$ is given by $(T, M^*)$, where $M^*$ is the function sending $t$ to $(M(t))^*$. For a subset $P$ of the ground set, the tree of matroids $\Tcal/P$ obtained from  $\Tcal$ by {\em contracting} $P$ is given by $(T, M/P)$, where $M/P$ is the function sending $t$ to $M(t)/(P \cap E(t))$. For a subset $Q$ of the ground set, the tree of matroids $\Tcal\backslash Q$ obtained from  $\Tcal$ by {\em deleting} $Q$ is given by $(T, M \backslash Q)$, where $M \backslash Q$ is the function sending $t$ to $M(t) \backslash (Q \cap E(t))$. 
\end{dfn}

The following lemma describes some basic properties of trees of matroids under duality, contraction and deletion. The proof is trivial from the definitions and hence is omitted.

\begin{lem}
For any tree $\Tcal$ of matroids, $\Tcal = \Tcal^{**}$. For any disjoint subsets $P$ and $Q$ of the ground set of $\Tcal$ we have $(\Tcal / P)^* = \Tcal^* \backslash P$, $(\Tcal \backslash Q)^* = \Tcal^* / Q$ and $\Tcal / P \backslash Q = \Tcal \backslash Q /P$. If $\Tcal$ has overlap 1 and $(\Tcal, \Psi)$ induces a matroid $M$, then $(\Tcal/P\backslash Q, \Psi)$ induces the matroid $M /P \backslash Q$ and $(\Tcal^*, \Psi\ct)$ induces the matroid $M^*$. \qed
\end{lem}

We will sometimes use the expression {\em $\Psi\ct$-cocircuits of $\Tcal$} for the $\Psi\ct$-circuits of $\Tcal^*$. If there is a matroid whose circuits are the  $\Psi$-circuits of $\Tcal$ and whose cocircuits are the $\Psi\ct$-cocircuits of $\Tcal$ then we will call that matroid the {\em $\Psi$-matroid} for $\Tcal$, and denote it $M_{\Psi}(\Tcal)$.

Note that if $\Tcal$ is finite then we always get a matroid in this way, which is simply the 2-sum of the matroids $M(t)$ in the sense of \cite{Oxley}.

\begin{lem}[Lemma 5.5,~\cite{BC:psi_matroids}]\label{2SumsO1}
Let $\Tcal = (T, M)$ be a tree of matroids, $\Psi$ a set of ends of $T$, and let $(C, o)$ and $(D, b)$ be respectively a $\Psi$-precircuit of $\Tcal$ and a $\Psi\ct$-precircuit of $\Tcal^*$. Then $|\underline{(C, o)} \cap \underline{(D, b)}| \neq 1$.
\end{lem}

\begin{dfn}
 If $T$ is a tree, and $tu$ is a (directed) edge of $T$, we take $T_{t \to u}$ to be the connected component of $T - t$ that contains $u$. If $\Tcal = (T, M)$ is a tree of matroids, we take $\Tcal_{t \to u}$ to be the tree of matroids $(T_{t \to u}, M \restric_{T_{t \to u}})$.
\end{dfn} 

So far we have discussed how to construct matroids by gluing together trees of `smaller' matroids. Now we turn to a notion, taken from~\cite{ADP:decomposition}, of a decomposition of a matroid into a tree of such smaller parts.

\begin{dfn}
A {\em tree-decomposition of adhesion 2} of a matroid $N$ consists of a tree $T$ and a partition $R = (R(v))_{v \in V(T)}$ of the ground set $E$ of $N$ such that for any edge $tt'$ of $T$ the partition $(\bigcup_{v \in V(T_{t \to t'})} R(v), \bigcup_{v \in V(T_{t' \to t})} R(v))$ is a 2-separation of $N$.

Given such a tree-decomposition, and a vertex $v$ of $T$, we define a matroid $M(v)$, called the {\em torso} of $T$ at $v$, as follows: the ground set of $M(v)$ consists of $R(v)$ together with a new element $e(vv')$ for each edge $vv'$ of $T$ incident with $v$. For any circuit $o$ of $N$ not included in any set $\bigcup_{t \in V(T_{v \to v'})} R(v)$, we have a circuit $\hat o(v)$ of $M(v)$ given by $(o \cap R(v) )\cup \{e(vv') \in E(v) | o \cap \bigcup_{t \in V(T_{v \to v'})} R(v)\neq \emptyset\}$. These are all the circuits of $M(v)$.

In this way we get a tree of matroids $\Tcal(N, T, R) = (T, v \mapsto M(v))$ of overlap 1 from any tree-decomposition of adhesion 2. For any circuit $o$ of $N$ we get a corresponding precircuit $(S_o, \hat o)$, where $S_o$ is just the subtree of $T$ consisting of those vertices $v$ for which $\hat o(v)$ is defined. We shall refer to this precircuit as the {\em canonical precircuit} of $o$.
\end{dfn}

Note that $\underline{(S_o, \hat o)} = o$. It is shown in~\cite[\S4, \S8]{ADP:decomposition} that each $M(v)$ really is a matroid, and we will show in \autoref{torso_minor} that it is isomorphic to a minor of $N$, and that $\Tcal(N^*, T, R) = (\Tcal(N, T, R))^*$.
\cite{ADP:decomposition} also contains the following theorem.

\begin{thm}[Aigner-Horev, Diestel, Postle]\label{deco}
 For any matroid $N$ there is a tree-decomposition $\Dcal(N)$ of adhesion 2 of $N$ such that all torsos have size at least 3 and are either circuits, cocircuits or 3-connected, and in which no two circuits and no two cocircuits are adjacent in the tree. This decomposition is unique in the sense that any other tree-decomposition with these properties must be isomorphic to it.
\end{thm}

The above theorem is a generalisation to infinite matroids of a standard result about finite matroids~\cite{findec1, findec2}. If $N$ is a finite matroid, it is possible to reconstruct $N$ from the decomposition $\Dcal(N)$. However, as noted in the introduction, it is not in general possible to reconstruct $N$ from $\Dcal(N)$ if $N$ is infinite.

\subsection{Tree-decompositions and $2$-separations in infinite matroids}

Our aim in this subsection is to show that if $N$ is a matroid with a tree-decomposition of adhesion 2 over some tree $T$ and $S$ is a subtree of $T$ then the torso of $N$ when we cut off the parts of $N$ not corresponding to $S$ is a minor of $N$. In particular, all the torsos mentioned in \autoref{deco} are minors of $N$, a statement which is claimed but not yet proved there. We begin by recalling that the familiar fact that 2-separations can be encoded by 2-sums extends straightforwardly to infinite matroids.

\begin{comment}
\begin{dfn}
 Let $N$ be a matroid with a tree-decomposition $(T,R)$ of adhesion 2,
and let $o$ be an $N$-circuit.
The \emph{canonical precircuit of $o$} is the pair $(S_o,\hat o)$ 
where $S_o$ is the smallest subtree of $T$ connecting all those nodes $t$
for which $o\cap R(t)$ is nonempty, and $\hat o$ is a function assigning each $t\in S_o$
an $M(t)$-circuit obtained from $o\cap R(t)$ by adding all virtual elements $e(tt')$ for $tt'\in S_o$.
\end{dfn}
\end{comment}

\begin{lem}[\cite{ADP:decomposition}, Lemma 3.6]\label{gluing2}
Let $(A,B)$ be a $2$-separation in a matroid $N$.
Then there are two matroids $N(A)$ and $N(B)$ with 
ground sets $A+e(A,B)$ and $B+e(A,B)$ where $e(A,B)$ is not in $E(N)$,
so that $N=N(A)\oplus_2 N(B)$. 
\end{lem}

\begin{comment}
 
We do not need the following Lemma

\begin{lem}\label{gluing}[Lemma 3.8 in TODO cite matroid decomposition]
If $o_1$ and $o_2$ are circuits of $N$ having elements on both sides of a $2$-separation 
$(A,B)$ of $N$, then $(o_1\cap A)\cup (o_2\cap B)$ is a circuit of $N$. 
\end{lem}

TODO define matroid $N(A)$ and $N(B)$ with virtual element $e(A,B)$.

Moreover, $N(A)$ and $N(B)$ are both isomorphic to a minor of $N$.

Let $(A,B)$ be a $2$-separation in a matroid $N$.
Let $o_A$ be a circuit of $N(A)$ using the element $e(A,B)$
and $o_B$ be a circuit of $N(B)$ using the element $e(A,B)$.
Then $o_A\Delta o_B=o_A\cup o_B-e(A,B)$ is an $N$-circuit.
Let $N$ be a matroid with a tree-decomposition of adhesion 2. 
\end{lem}

\end{comment}

\begin{dfn}
Let $(T, R)$ be a tree-decomposition of a matroid $N$ of adhesion 2, and let $S$ be a subtree of $T$. The {\em star-decomposition} $(T_S, R_S)$ of $N$ corresponding to $S$ is given by taking $T_S$ to be the star with central node $S$ and with a leaf $l(tt')$ for each edge $tt'$ of $T$ with $t \in S$ but $t' \not \in S$, and taking $R_S(S) = \bigcup_{v \in S}R(v)$ and $R_S(l(tt')) = \bigcup_{v \in T_{t \to t'}} R(v)$. The star-decomposition also has adhesion 2.

We define $N_S$ to be the torso of $T_S$ at $S$. For notational convenience, we will identify the virtual element $e(Sl(tt'))$ of this torso with the virtual element $e(tt')$ arising in the construction of $\Tcal(N, T, R)$. There is a natural tree-decomposition $(S, R \restric_S)$ of $N_S$, where we take $R \restric_S(t)$ to consist of $R(t)$ together with all the virtual elements $e(tt')$ with $t'$ a neighbour of $t$ in $T$ such that $t' \not \in S$.
\end{dfn}

By~\cite{ADP:decomposition}, Lemma 4.9, the tree-decomposition $(S, R \restric_S)$ has adhesion 2. It is clear that $\Tcal(N_S, S, R \restric_S)$ is just the restriction of $\Tcal(N, T, R)$ to $S$, and that the $N_S$-circuits are precisely the 
nonempty underlying sets of canonical precircuits of $N$-circuits restricted to $S$.
(The dual statement is true for the $N'$-cocircuits.)

We say that two matroids are \emph{realistically isomorphic} if 
there is an isomorphism between them that is the identity when restricted to the intersection of their ground sets.

We are finally in a position to prove the main result of this subsection.

\begin{thm}\label{X1}
Let $N$ be a matroid with a tree-decomposition $(T,R)$ of adhesion 2.
Let $S$ be a subtree of $T$.
Then $N_S$ is realistically isomorphic to a minor of $N$.
\end{thm}

\begin{proof}
Without loss of generality, we may assume that $N$ is connected.
In particular all the separations associated to edges of $T$ are exact $2$-separations.
Let $\partial S$ be the set of those nodes 
$x$ of $T$ that have a neighbour $t(x)$ in $S$ but are not in $S$.

First we define a minor $N'$ of $N$ from which we shall define a realistic isomorphism to $N_S$.
%Let $X = \bigcup_{v \in S}R(v)$. 
%For $x\in \partial S$, let $Y_x$ be the set of those edges of $N$ that are in 
%$M(t)$ with $t\in T_{t(x)\to x}$.

For each $x \in \partial S$, let $s(x)$ be a base of $N$ contracted onto $R_S(l(t(x)x))$. 
Since the separation associated to the edge $t(x)x$ is an exact 2-separation,
we can add a single element $e(x)$ to $s(x)$
to make it a base of $N$ restricted onto $R_S(l(t(x)x))$.

%Let $s_1(x)$ be a base of $N$ restricted onto $Y_x$ extending $s(x)$.
%Since the separation associated to the edge $t(x)x$ is an exact 2-separation
%$s_1(x)\sm s(x)$ contains precisely one element which we call $e(x)$.
Let $K$ be the set of all those $e(x)$.
%We obtain $\tilde N$ from $N$ by contracting all the sets $s(x)$.
% we obtain $N''$ from $\tilde N$ by restricting onto $X\cup K$.
We obtain $N'$ from $N$ by contracting all 
the sets $s(x)$ and then restricting onto $R_S(S)\cup K$.
Let the matroid $N''$ be obtained from $N'$ by replacing the element $e(x)$ by the virtual element $e(t(x)x)$. This defines a realistic isomorphism $\alpha: N''\to N'$. We will now show that $N_S = N''$.
%It is worth noting that in $\tilde N$ 
%any two possible choices for $e(x)$ land up being in parallel so that $N'$ does not depend on %the choices of the $e(x)$.
%
%\vspace{0.3 cm}
%
%Let $\Ccal$ be the set of nonempty underlying sets of canonical precircuits of $N$-circuits restricted to $S$, and let $\Dcal$ be the set of nonempty underlying sets of canonical precocircuits  of $N$-cocircuits restricted to $S$. 
%
%Having defined $N'$, it remains to show that $\Ccal$ is the set of $N'$-circuits and $\Dcal$
%is the set of $N'$-cocircuits. 
This will be done using \autoref{cir_c_scrawl} applied to $N''$, with $\Ccal$ the set of $N_S$-circuits  and $\Dcal$ the set of $N_S$-cocircuits.
So we first have to check that the assumptions of this Lemma are satisfied. The first assumption is proved as follows:
\begin{claim}
Every $N''$-circuit is an element of $\Ccal$.
\end{claim}
\begin{proof}
Let $o''$ be an $N''$-circuit. 
Let $o'=\alpha(o'')$.
%Let $A(o')$ be the set of those $e(x)$ such that $e(t(x)x)\in o'$. Let $o''$ be the set of %those elements of $o'$ that are in $E(N)$ together with $A(o')$.
%Clearly, $o''$ is an $N''$-circuit, 
Then $o'$ extends to an $N$-circuit $o$ using additionally only elements of sets $s(x)$ by \autoref{rest_cir}.

It remains to show for each $e(x)\notin o$ that $o$ does not use any element of $s(x)$, 
since then the underlying set of the canonical precircuit of $o$ restricted to $S$ will be $o''$.
Now $o\cap R_S(l(t(x)x))$ must be a union of circuits of the contraction $N.R_S(l(t(x)x))$ of $N$ onto $R_S(l(t(x)x))$. For $e(x)\notin o$ the set $o\cap R_S(l(t(x)x))$ must be empty as $o\cap R_S(l(t(x)x)) \se s(x)$ which is independent in that contraction. Thus every $N''$-circuit is the underlying set of the canonical precircuit of some $N$-circuit restricted to $S$.
\end{proof}

The dual argument shows that every $N''$-cocircuit is the underlying set of the canonical precircuit of some $N$-cocircuit restricted to $S$.

The last assumption of \autoref{cir_c_scrawl}, that no $o'\in \Ccal$ meets some $b'\in \Dcal$ in a single element $e$, is true since $N_S$ is a matroid. Applying \autoref{cir_c_scrawl} now gives that the $N''$-circuits are the minimal nonempty elements of $\Ccal$, which is the same as saying that they are the elements of $\Ccal$ since $\Ccal$ is the set of circuits of the matroid $N_S$, and no circuit of a matroid is included in any other. Similarly, the $N''$-cocircuits are the elements of $\Dcal$. This completes the proof.

\end{proof}

\begin{cor}\label{torso_minor}
Let $N$ be a matroid with a tree-decomposition of adhesion 2.
Then every torso $M(v)$ of this tree-decomposition is isomorphic to a minor of $N$.
\end{cor}

\begin{proof}
 Apply \autoref{X1} in the case that $S$ consists only of the node $v$.
\end{proof}

\section{Characterising the matroids arising from a nice ray of matroids}\label{sec:ray_case}

For this section, we fix a ray $\Rcal = (M_i | i \in \Nbb)$ of matroids of overlap 1. We call the end of this ray $\omega$. 
We refer to the element shared by $M(i)$ and $M(i + 1)$ as $e(i)$. We say that a matroid is an {\em $\Rcal$-matroid} if all of its circuits are circuits of $M_{\{\omega\}}(\Rcal)$ (in which the circuits are permitted to use the end) and all of its cocircuits are cocircuits of $M_{\emptyset}(\Rcal)$ (in which the circuits are not permitted to use the end). Our aim will be to give an explicit description of the collection of $\Rcal$-matroids. However, before we can do this we will impose a mild condition to exclude examples like the one below.

\begin{eg}
Here the tree $T$ is a ray and each $M(t)=M(C_4)$, arranged as in \autoref{fig:nice_needed}.
Then $M_\emptyset(\Tcal)$ is the free matroid but $M_{\Omega(T)}(\Tcal)$
consists of a single infinite circuit. So any pair of elements forms an $M_{\Omega(T)}(\Tcal)$-cocircuit which is not an $M_\emptyset(\Tcal)$-cocircuit.

 \begin{figure}
\begin{center}
 \includegraphics[width=8cm]{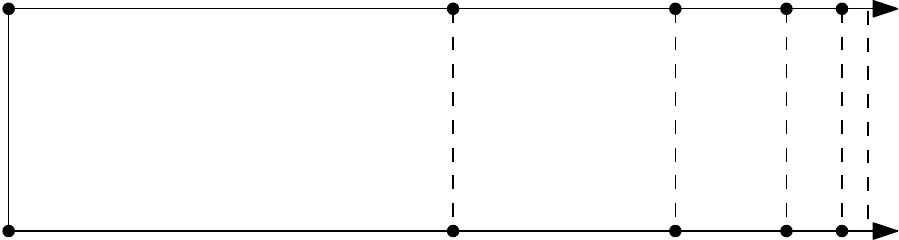}
\end{center} 
\caption{A badly behaved tree of matroids}\label{fig:nice_needed}
\end{figure}
\end{eg}

\begin{dfn}
A precircuit $(S, o)$ for a tree $\Tcal = (T, M)$ of matroids of overlap 1 is called a {\em phantom precircuit} if there is an edge $tt'$ of $S$ such that $o(v) \cap E(\Tcal) = \emptyset$ for $v \in V(S_{t \to t'})$.

$\Tcal=(T,M)$ is \emph{nice} if neither $\Tcal$ nor $\Tcal^*$ has any phantom precircuits.
\end{dfn}
Informally, a phantom precircuit is one which invisibly uses the nodes beyond some edge $tt'$, in that its supporting tree goes beyond that edge, but in such a way that this is not visible from its underlying set.

Note that $\Tcal=(T,M)$ is nice iff there is no $tt'\in E(T)$ such that
in $\Tcal_{t \to t'} = (T_{t \to t'}, M \restric_{V(T_{t \to t'})})$ the element $e(tt')$ is either a loop in $M_{\Omega(\Tcal_{t \to t'})}(\Tcal_{t \to t'})$ 
or a coloop $M_{\emptyset}(\Tcal_{t \to t'})$.

\begin{lem}\label{ray+td_nice}
 Let $N$ be a matroid with a tree-decomposition $(T,R)$ of adhesion 2.
\begin{enumerate}
 \item For every $N$-circuit $o$ its corresponding precircuit $(S_o,\hat o)$ is not phantom.
\item If $T$ is a ray, and there are a circuit $o$ and a cocircuit $b$ of $N$ that both have elements in infinitely many of the $R(v)$, then $\Tcal(N,T,R)$ is nice.
\end{enumerate}
\end{lem}

\begin{proof}
(1) follows from the definition of $S_o$.

For (2), let $T=t_1,t_2,\ldots$ be a ray. 
Now suppose for a contradiction that 
there is a phantom precircuit $(S_c,c)$. Then for all sufficiently large $n$,
the circuit $c(t_n)$ consists of $e(t_{n-1}t_n)$ and $e(t_{n}t_{n+1})$.
In other words, $e(t_{n-1}t_n)$ and $e(t_{n}t_{n+1})$ are in parallel.

So $c(t_n)\se \hat o(t_n)$, hence $c(t_n)= \hat o(t_n)$.
This contradicts (1).
The case that there is a phantom precocircuit $(S_c,c)$ is dual.
Hence $\Tcal(N,T,R)$ is nice.
\end{proof}

\begin{lem}\label{cir_in_omega_cir}
Let $\Tcal=(T,M)$ be a nice tree of matroids, then every $\emptyset$-circuit is
an $\Omega(T)$-circuit. Dually, every $\emptyset$-cocircuit is an $\Omega(T)$-cocircuit.
\qed
\end{lem}

\begin{lem}\label{cir_in_cir}
Let $\Tcal=(T,M)$ be a nice tree of matroids, and $N$ be a matroid such that 
$\Ccal(N)\se \Ccal(M_{\Omega(T)}(\Tcal))$ and  $\Ccal(N^*)\se \Ccal(M^*_{\emptyset}(\Tcal))$.

Then $\Ccal(M_{\emptyset}(\Tcal))\se \Ccal(N)$ and  $\Ccal(M^*_{\Omega(T)}(\Tcal))\se \Ccal(N^*)$.
\end{lem}

\begin{proof}
 By duality, it suffices to prove only that $\Ccal(M_\emptyset(T))\se \Ccal(N)$.
So let $o\in \Ccal(M_\emptyset(T))$. 
Since $o$ never meets an element of $\Ccal(M^*_{\emptyset}(T))$ just once,
it never meets an element of $\Ccal(N^*)$ just once.
Hence $o$ includes an $N$-circuit $o'$ by \autoref{is_scrawl}.
Thus $o'\in  \Ccal(M_{\Omega(T)}(T))$. By \autoref{cir_in_omega_cir},
we must have $o'=o$. So $o\in \Ccal(N)$, as desired.
\end{proof}

From now on, we shall assume that $\Rcal$ is a nice ray of matroids, and that $N$ is an $\Rcal$-matroid. We shall refer to the nodes of the supporting ray of $\Rcal$ as $t_1, t_2, \ldots$. For brevity, we shall denote $M_{t_i}$ by $M_i$ and $e(t_i)$ by $e(i)$ for any $i \in \Nbb$.

We say that a subset of $E(\Rcal)$ is {\em prolonged} if it meets infinitely many of the sets $E(M_i)$ with $i \in \Nbb$.

We begin by defining a fundamental equivalence relation on the potential prolonged circuits and cocircuits of $N$. Let $\Ccal_{\infty}$ be the set of prolonged elements of $\Ccal(M_{\{\omega\}}(\Rcal))$ and $\Dcal_{\infty}$ be the set of prolonged elements of $\Ccal(M^*_{\emptyset}(\Rcal))$. We define a relation $\sim$ from $\Ccal_{\infty}$ to $\Dcal_{\infty}$ by $o \sim b$ if and only if $o \cap b$ meets only finitely many of the sets $E(M_i)$. We can think of $\sim$ as a relation on the set $\Ccal_{\infty} \sqcup \Dcal_{\infty}$, and we define $\simeq$ to be the equivalence relation generated by this relation. The relation $\simeq$ restricts to equivalence relations on each of $\Ccal_{\infty}$ and $\Dcal_{\infty}$; we denote the equivalence class of $o$ in $\Ccal_{\infty}/\simeq$ by $[o]$ and the equivalence class of $b$ in $\Dcal_{\infty}/\simeq$ by $[b]$.

The first important fact about this relation is that the set of prolonged circuits of $N$ must be closed under $\simeq$ (within $\Ccal_{\infty}$). This follows from the following lemmas.

\begin{lem}\label{finfix}
Let $o$ be a prolonged circuit of $N$, and let $o' \in \Ccal_{\infty}$ such that their symmetric difference $o \triangle o'$ meets only finitely many of the sets $E(M_i)$. Then $o'$ is also a circuit of $N$.
\end{lem}
\begin{proof}
Let the circuits $o$ and $o'$ be represented by the precircuits $(I, \hat o)$ and $(I', \hat o')$.

We consider first of all the case where $o \triangle o'$ is a subset of $E(M_1)$. We apply Lemma \ref{good_cir_eli} to $\hat o(1)$ and $\hat o'(1)$, obtaining a set $X \subseteq \hat o(1) - e(1)$ and a family $(o_x | x \in X)$ of circuits of $M_1$ with $o_x \cap (X \cup \{e(1)\}) = \{x\}$ for each $x \in X$ such that the only circuit $\bar o$ with $e(1) \in \bar o \subseteq \hat o(1) \cup (\bigcup_{x \in X} o_x) \setminus X$ is $\hat o'(1)$. Each of the $o_x$ is a circuit of $M_{\emptyset}(\Rcal)$, and so of $N$. Now we apply the circuit elimination axiom (C3) in $N$ to the circuit $o$ and the circuits $o_x$, eliminating the set $X$ and keeping some $e \in o \cap E(M_n)$ with $n > 1$. We obtain an 
$N$-circuit $o''$ with $e \in o'' \subseteq o \cup (\bigcup_{x \in X} o'_x) \setminus X$.

Let $o''$ be represented by the precircuit $(I'', \hat o'')$. For any $i > 1$ with $i \in I''$, we have $\{e(i-1), e(i)\} \subseteq \hat o(i)$ and so $\hat o''(i) \subseteq \hat o(i)$, which implies that $\hat o''(i) = \hat o(i)$ and in particular $\{e(i-1), e(i)\} \subseteq \hat o''(i)\}$ so that both $i - 1$ and $i + 1$ must be in $I''$. So since $I''$ contains $n$, it must contain all natural numbers. Furthermore, since $e(k) \in \hat o''(k) \subseteq \hat o(k) \cup (\bigcup_{x \in X} o_x) \setminus X$ we must have $\hat o''(k) = \hat o'(k)$. Thus $o'' \triangle o' \subseteq \bigcup_{i < k} E(M_i)$.
This completes the proof of the special case. 

We can now reduce the more general statement of the lemma to the special case above as follows: for any $k$, we can obtain a new ray of matroids whose first element is the 2-sum of the first $k$ matroids along $\Rcal$ and where the $i$\textsuperscript{th}\ matroid with $i > 1$ is $M_{i - k + 1}$. Since $\Ccal_{\infty}$ is preserved under this operation, we can apply it to produce a new ray of matroids where the ground set of the first element includes $o \triangle o'$. Arguing as above in this ray of matroids, we are done.
\end{proof}

\begin{lem}\label{cir_closed}
Let $o$ be a prolonged circuit of $N$, and let $b \in \Dcal_{\infty}$ and $o' \in \Ccal_{\infty}$ such that $o \sim b$ and $o' \sim b$. Then $o'$ is also a circuit of $N$.
\end{lem}
\begin{proof}
Let the circuits $o$ and $o'$ be represented by the precircuits $(I, \hat o)$ and $(I', \hat o')$, and let the cocircuit $b$ be represented by the precocircuit $(J, \hat b)$.
Choose some $e \in o \cap E(M_n)$ for some $n$, and choose some $k > n$ such that for $i \geq k$ we have $\hat o(i) \cap \hat b(i) = \hat o'(i) \cap \hat b(i) = \{e(i-1), e(i)\}$. For each $i \geq k$, we may by Lemma \ref{good_cir_eli} find a set $X_i \subseteq \hat o(i) - e(i-1)$ and a family $(o_x | x \in X_i)$ of circuits of $M_i$ with $o_x \cap (X_i \cup \{e(i-1)\}) = \{x\}$ for each $x \in X_i$ such that the only circuit $o$ with $e(i-1) \in o \subseteq \hat o(i) \cup (\bigcup_{x \in X_i} o_x) \setminus X_i$ is $\hat o'(i)$. Note that for $x \in X_i$ we have $o_x \cap \hat b(i) \subseteq \{e(i)\}$, so that $o_x \cap \hat b(i) = \emptyset$ and in particular $e(i) \not \in o_x$. So $o_x$ is a circuit of $M_{\emptyset}(\Rcal)$, and so of $N$. Now we apply the circuit elimination axiom (C3) in $N$ to the circuit $o$ and the circuits $o_x$ with $x \in \bigcup_{i \geq k}X_i$, eliminating the set $\bigcup_{i \geq k}X_i$ and keeping the element $e$. We obtain an $N$-circuit $o''$ with $e \in o'' \subseteq o \cup (\bigcup_{i \geq k}\bigcup_{x \in X_i} o_x) \setminus \bigcup_{i \geq k}X_i$.

Let $o''$ be represented by the precircuit $(I'', \hat o'')$. Then for any $i < k$ in $I''$ we have $\hat o''(i) \subseteq \hat o(i)$, and so $\hat o''(i) = \hat o(i)$, so that all neighbours of $i$ in $I$ are in $I''$. Thus since $n \in I''$ we get $I'' \cap \{1, .., k\} = I \cap \{1, .., k\}$. Now we show by induction on $i$ that if $i \geq k$ then $i \in I$ and $\hat o''(i) = \hat o'(i)$. We begin by noting that $i \in I$ and $e(i-1) \in \hat o''(i)$. For the base case this follows from the fact that $\hat o''(k-1) = \hat o(k-1)$ and otherwise it follows from the induction hypothesis. So we have $e(i-1) \in \hat o''(i) \subseteq \hat o(i) \cup (\bigcup_{x \in X_i} o_x) \setminus X_i$, which implies that $\hat o''(i) = \hat o'(i)$ as required. Thus $o'' \triangle o' \subseteq \bigcup_{i < k} E(M_i)$ and so applying Lemma \ref{finfix} we get that $o'$ is a circuit of $N$.
\end{proof}

Thus the set of circuits of $N$ must consist of $\Ccal(M_{\emptyset}(\Rcal))$ together with a union of some $\simeq$-equivalence classes\footnote{As usual, we denote the union of the elements of a set $S$ of equivalence classes by $\bigcup S$.} in $\Ccal_{\infty}$. In fact, we can show that this is the only restriction.

\begin{thm}\label{ray_case}
Let $\Phi$ be any subset of $\Ccal_{\infty} / \simeq$. Then there is an $\Rcal$-matroid $M_{\Phi}(\Rcal)$ whose set of circuits is $\Ccal(M_{\emptyset}(\Rcal)) \cup \bigcup \Phi$. The set of cocircuits of this matroid is $\Ccal(M^*_{\{\omega\}}(\Rcal)) \cup \bigcup \Phi^*$, where $\Phi^* \subseteq \Dcal_{\infty}/\simeq$ is the set of equivalence classes $[b]$ such that there is no $[o] \in \Phi$ with $o \simeq b$. Conversely, the set of circuits of any $\Rcal$-matroid has the form $\Ccal(M_{\emptyset}(\Rcal)) \cup \bigcup \Phi$ with $\Phi$ a subset of $\Ccal_{\infty} /\simeq$.
\end{thm}

The remainder of this section will be devoted to proving this theorem. The last sentence is just a restatement of Lemma \ref{cir_closed}, and it is clear that if this construction defines a matroid then that matroid is an $\Rcal$-matroid. So we just have to show that this construction gives a matroid. We will do this using Theorem \ref{hybrid}. We must now show that the conditions of that theorem hold. (C0) and (C1) for $\Ccal(M_{\emptyset}(\Rcal)) \cup \bigcup \Phi$ follow from the same axioms applied to $\Ccal(M_{\{\omega\}}(\Rcal))$, and we get (C0) and (C1) for $\Ccal(M^*_{\{\omega\}}(\Rcal)) \cup \bigcup \Phi^*$ from the same axioms applied to $\Ccal(M^*_{\emptyset}(\Rcal))$. Checking the remaining conditions is the purpose of the following lemmas.
\begin{lem}[O1]
There do not exist $o \in \Ccal(M_{\emptyset}(\Rcal)) \cup \bigcup \Phi$ and $b \in \Ccal(M^*_{\{\omega\}}(\Rcal)) \cup \bigcup \Phi^*$ with $|o \cap b| = 1$. 
\end{lem}
\begin{proof}
Using (O1) in $M_{\emptyset}(\Rcal)$ and $M_{\{\omega\}}(\Rcal)$, we see that the only way this could be possible is if $o \in \bigcup \Phi$ and $b \in \bigcup \Phi^*$. But then $o \sim b$, contradicting the definition of $\Phi^*$.
\end{proof}

\begin{lem}[O2]\label{R_O2}
Let $E=P\dot\cup Q \dot\cup \{e\}$ be a partition of $E$, with $e \in E(M_n)$. Then
either $P+e$ includes an element of $\Ccal(M_{\emptyset}(\Rcal)) \cup \bigcup \Phi$ containing $e$ or
$Q+e$ includes an element of $\Ccal(M^*_{\{\omega\}}(\Rcal)) \cup \bigcup \Phi^*$ containing $e$.
\end{lem}
\begin{proof}
Suppose for a contradiction that neither of these options holds. Applying (O2) in the matroid $M_{\emptyset}(\Rcal)$ we obtain either $o \in \Ccal(M_{\emptyset}(\Rcal))$ with $e \in o \subseteq P + e$ or $b \in \Ccal(M^*_{\emptyset}(\Rcal))$ with $e \in b \subseteq Q + e$. The first of these is impossible by assumption, so there is such a $b$. Similarly, applying (O2) in the matroid $M_{\{\omega\}}(\Rcal)$ we get a circuit $o$ of that matroid with $e \in o \subseteq Q + e$. Furthermore, by our assumptions $o \in \Ccal_{\infty}$ and $b \in \Dcal_{\infty}$. Since $o \cap b = \{e\}$, we have $o \sim b$, and so either $[o] \in \Phi$ or else $[b] \in \Phi^*$, as required.
\end{proof}

\begin{lem}
$\Ccal(M_{\emptyset}(\Rcal)) \cup \bigcup \Phi$ satisfies (CM).
\end{lem}
\begin{proof}
Let $I$ be a set not including any element of $\Ccal(M_{\emptyset}(\Rcal)) \cup \bigcup \Phi$, and $X$ a set with $I \subseteq X \subseteq E$. Since $I$ is $M_{\emptyset}(\Rcal)$-independent, we may apply $(CM)$ in that matroid to extend it to a maximal $M_{\emptyset}(\Rcal)$-independent subset $J$ of $X$. If $J$ fails to include any element of $\bigcup \Phi$ then we are done, so suppose that it does include such an element $o$. Let $e$ be any element of $o$ not contained in $I$, and let $J' = J - e$. We will show that $J'$ is maximal amongst those subsets of $X$ including $I$ but not including any element of $\Ccal(M_{\emptyset}(\Rcal)) \cup \bigcup \Phi$, which will complete the proof of (CM). It is clear that $I \subseteq J' \subseteq J$.

First we show that $J'$ includes no element of $\Ccal(M_{\emptyset}(\Rcal)) \cup \bigcup \Phi$. Suppose for a contradiction that it does include such an element $o'$. Since $J'$ is a subset of the $M_{\emptyset}(\Rcal)$-independent set $J$ we must have $o' \in \bigcup \Phi$. Now suppose that $e \in E(M_i)$, let $o = \underline{(S_o, \hat o)}$ and let $o' = \underline{(S_{o'}, \hat o')}$. Choose $j > i$ with $j \in S_o \cap S_{o'}$, and such that $e(j+1), e(j+2)$ is not a cocircuit of $M_{j+1}$ (this is possible because $\Rcal$ is nice). Let $\overline o$ be a circuit of $M_{j+1}$ with $e(j+1) \in \overline o$ but $e(j+2) \not \in \overline o$. Then we can build an $\emptyset$-precircuit $(S_o \cap \{k | k \leq j + 1\}, \hat c)$ with $\hat c(j + 1) = \overline o$ and $\hat c (k) = \hat o(k)$ for $k \leq j$. We take $c = \underline{(S_o \cap \{k | k \leq j + 1\}, \hat c)}$. We define an $\emptyset$-circuit $c'$ in a similar way from $o'$. Applying circuit elimination in $M_{\emptyset}$ to $c$ and $c'$, keeping 
$e$ and eliminating some element of $\overline o$, we obtain an $\emptyset$-circuit $o''$ with $e \in o'' \subseteq J \cup \overline o$. But since $o'' \cap \overline o$ is a proper subset of $\overline o - e(j+1)$, we must have $o'' \cap \overline o = \emptyset$, and so $o'' \subseteq J$, which contradicts the $M_{\emptyset}$-independence of $J$.

Next we show that $J'$ is maximal. Suppose for a contradiction that there is some $e' \in X \setminus J'$ such that $J' + e'$ includes no element of $\Ccal(M_{\emptyset}(\Rcal)) \cup \bigcup \Phi$. By maximality of $J$, there must be some $M_{\emptyset}$-circuit $o'$ with $e' \in o' \subseteq J + e'$. Since $o' \not \subseteq J'$, we have $e \in o'$. By Lemmas \ref{cireli_partitioning} and \ref{R_O2} we know that $\Ccal(M_{\emptyset}(\Rcal)) \cup \bigcup \Phi$ satisfies circuit elimination. We apply this to $o'$ and $o$, keeping $e'$ and eliminating $e$. This gives an element $o''$ of $\Ccal(M_{\emptyset}(\Rcal)) \cup \bigcup \Phi$ included in $J' + e'$, which is the desired contradiction.
\end{proof}

\section{Representability properties of $M_{\Phi}(\Qcal)$}\label{repQ}
\begin{figure}
\begin{center}
 \includegraphics[width=8cm]{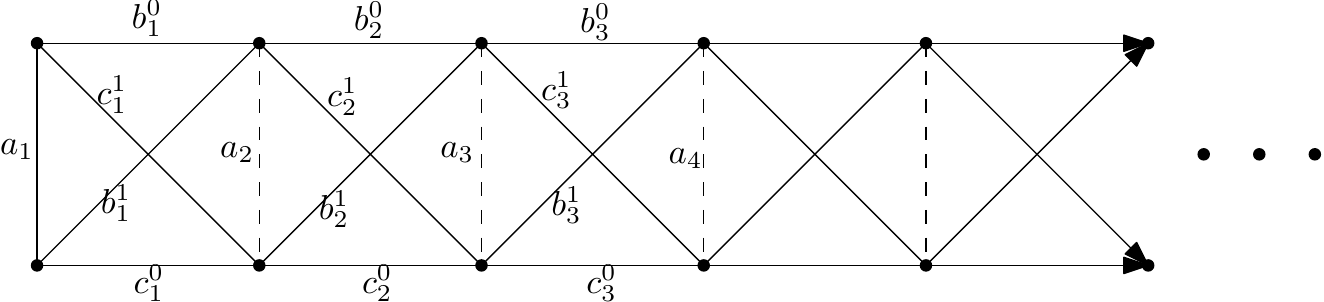}
\end{center} 
\caption{The ray $\Qcal$ of matroids}\label{fig:Qray}
\end{figure}
Recall from the introduction the ray $\Qcal = (K_4 | i \in \Nbb)$ of matroids, in which the edge set of the $i$\textsuperscript{th} copy of $K_4$ is $\{a_i, b^0_i, b^1_i, c^0_i, c^1_i, a_{i + 1}\}$, as shown in figure \autoref{fig:Qray}. The results of the last section give us a simple characterisation of the set of $\Qcal$-matroids. In this section, we shall investigate the extent to which these $\Qcal$-matroids are binary. To make this more precise, we begin by recalling a theorem from~\cite{BC:rep_matroids} excluded minors method:

\begin{thm}\label{char_binary}
Let $M$ be a tame matroid. Then the following are equivalent:
\begin{enumerate}
\item $M$ is a binary thin sums matroid.
\item For any circuit $o$ and cocircuit $b$ of $M$, $|o \cap b|$ is even.
\item For any circuit $o$ and cocircuit $b$ of $M$, $|o \cap b| \neq 3$
\item $M$ has no minor isomorphic to $U_{2,4}$. 
\item If $o_1$, $o_2$ are circuits then $o_1 \triangle o_2$ is empty or includes a circuit.
\item If $o_1$, $o_2$ are circuits then $o_1 \triangle o_2$ is a disjoint union of circuits.
\item If $(o_i | i \in I)$ is a finite family of circuits then $\bigtriangleup_{i \in I}o_i$ is empty or includes a circuit.
\item If $(o_i | i \in I)$ is a finite family of circuits then $\bigtriangleup_{i \in I}o_i$ is a disjoint union of circuits.
\item For any base $s$ of $M$, and any circuit $o$ of $M$, $o = \bigtriangleup_{e \in o \setminus s} o_e$, where $o_e$ is the fundamental circuit of $e$ with respect to $s$.
\end{enumerate}
\end{thm}

However, if $M$ is not tame, these conditions are not equivalent. We shall illustrate this by considering which of the conditions hold for each $\Qcal$-matroid (most of which are not tame).

First we give a simple description of the set of $\Qcal$-matroids.

Let $\omega$ be the end of $\Qcal$. For any sequence $(v(i) \in \{0, 1\} | i \in \Nbb)$ there is an infinite circuit $o(0, v)$ of $M_{\{\omega\}}(\Qcal)$ consisting of $a_1$ and the $b_i^{v(i)}$ and $c_i^{v(i)}$ with $i \in \Nbb$. Also, for any $n \in \Nbb$, and any sequence $(v(i)| i \geq n)$ there is such an infinite circuit $o(n, v)$ consisting of $b^{v(n)}_n$, $c^{1-v(n)}_n$ and the $b_i^{v(i)}$ and $c_i^{v(i)}$ with $i > n$. A typical such circuit with $n = 2$ is shown in \autoref{fig:Qcircuit}. It is not hard to show that every infinite circuit arises in this way. Using the fact that $K_4$ is self-dual, we may also check that the sets $o(n, v)$ are precisely the infinite cocircuits of $M_{\emptyset}(\Qcal)$. It is clear that $o(n, v) \sim o(n', v')$ if and only if for all sufficiently large $i$ we have $v(i) \neq v'(i)$, so that the equivalence relation $\simeq$ on $\Ccal_{\infty}$ is given by $o(n, v) \simeq o(n', v')$ if and only if for all sufficiently large $i$ we have $v(i) = v'(i)$.

\begin{figure}
\begin{center}
 \includegraphics[width=8cm]{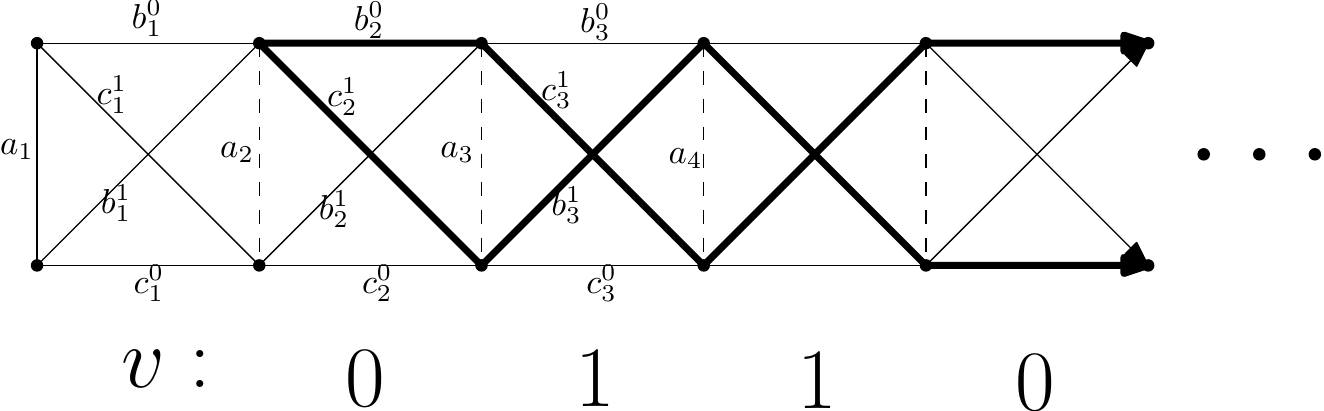}
\end{center} 
\caption{A $\Qcal$-circuit}\label{fig:Qcircuit}
\end{figure}

Applying Theorem \ref{main:thm_intro}, we obtain that for any set $\Phi$ of $\simeq$-equivalence classes, there is a matroid $M_{\Phi}(\Qcal)$ with circuit set $\Ccal(M_{\emptyset}(\Qcal)) \cup \bigcup \Phi$ and cocircuit set $\Ccal(M^*_{\{\omega\}}(\Qcal)) \cup \bigcup \Phi^*$, where $\Phi^* \subseteq \Dcal_{\infty} / \simeq$ is the set of equivalence classes $[b]$ such that there is no $[o] \in \Phi$ with $o \simeq b$. Furthermore, every $\Qcal$-matroid arises in this way.

We now proceed to analyse which of the conditions of Theorem \ref{char_binary} hold for these $\Qcal$-matroids. If $\Phi = \emptyset$ then $M_{\Phi}(\Qcal)$ is finitary, and so is tame. Furthermore, in this case condition (3) is easy to check, and applying the theorem, we obtain that all the conditions hold. Similar arguments apply to $M_{\Ccal_{\infty}/\simeq}(\Qcal)$, which is cofinitary. In the following list, we shall consider what happens for a typical $\Phi$ with $\emptyset \subsetneq \Phi \subsetneq \Ccal_{\infty}/\simeq$. Let $E$ be the common ground set of the matroids $M_{\Phi}(\Qcal)$.

It will turn out that in these nontrivial cases, we always have (3-6) and never have (1-2) or (9), but sometimes have (7) and (8).

\begin{enumerate}
\item We shall show that $M_{\Phi}(\Qcal)$ is not a binary thin sums matroid. Suppose for a contradiction that it is: then there is a set $A$ and a family of functions $(\phi_e \colon A \to \Fbb_2 | e \in E)$ such that the circuits of $M_{\Phi}(\Qcal)$ are the minimal nonempty supports of functions $c \colon E \to \Fbb_2$ such that for any $a \in A$ the sum $\sum_{e \in E} c(e) \phi_e(a)$ is well defined and equal (when evaluated in $\Fbb_2$) to 0. Equivalently, there is a set $\Dcal$ of subsets of $E$ (given by $\{\{e \in E | \phi_e(a) = 1\}| a \in A\}$) such that the circuits of $M_{\Phi}(\Qcal)$ are the minimal nonempty subsets of $E$ whose intersection with any element of $\Dcal$ is finite and of even size.

Since no intersection of any element of $\Dcal$ with any circuit of $M_{\Phi}(\Qcal)$ has size 1, each element of $\Dcal$ is a union of cocircuits of $M_{\Phi}(\Qcal)$. By our assumption that $\Phi \neq \Ccal_{\infty}/\simeq$, there must be some element $s$ of $\Dcal$ which is not a union of cocircuits of $M_{\{\omega\}}(\Qcal)$, and so which includes some element $b$ of $\bigcup \Phi^*$. Let $o$ be any element of $\bigcup \Phi$. Then $o \not \sim b$, and so $o \cap b$ is infinite, and so $o \cap s$ is infinite, which is the desired contradiction.

\item Let $o \in \bigcup \Phi$ and $b \in \bigcup \Phi^*$. Then $|o \cap b|$ is not even: it is infinite.

\item The only ways an intersection of a circuit $o$ with a cocircuit $b$ of $M_{\Phi}(\Qcal)$ can be finite are if $o \in \Ccal(M_{\emptyset}(\Qcal))$ or $b \in \Ccal(M^*_{\{\omega\}}(\Qcal))$, and in either case $|o \cap b|$ is even. In particular, we never have $|o \cap b| = 3$.

\item If $M_{\Phi}(\Qcal)$ had a minor isomorphic to $U_{2, 4}$, then in that minor there would be a circuit and a cocircuit with intersection of size 3. But then by \autoref{rest_cir} we could extend these to a circuit and a cocircuit of $M_{\Phi}(\Qcal)$ with intersection of size 3, and we have just shown that this is impossible. So there is no such minor.

\item Let $o_1$ and $o_2$ be circuits of $M_{\Phi}(\Qcal)$. Using the fact that (5) holds in $M_{\{\omega\}}(\Qcal)$, we obtain that $o_1 \triangle o_2$ includes some element $o_3$ of $\Ccal_{\infty}$. If both $o_1$ and $o_2$ are finite, so is $o_3$, so $o_3$ is a circuit of $M_{\Phi}(\Qcal)$. If $o_1$ is finite but $o_2$ is not, we have $o_2 \simeq o_3$, so that $o_3$ is again a circuit of $M_{\Phi}(\Qcal)$. The case where $o_2$ is finite but $o_1$ is not is similar. If both are infinite, and $o_3$ is also infinite, then for any $i$ bigger than all of $n_1$, $n_2$ and $n_3$ we get that  $o_3$ meets $\{b^0_i, b^1_i, c^0_i, c^1_i\}$, so $(o_1 \triangle o_2)$ also meets this set and so (by construction) must include it. But then we get a finite circuit $o_3' \subseteq o_1 \triangle o_2$ given by $(o_3 \cap \{a_1\} \cup \bigcup_{j < i}\{b^0_j, b^1_j, c^0_j, c^1_j\}) \cup \{b^0_i, b^1_i\}$ for any such $i$, and $o_3'$ is a circuit of $M_{\Phi}(\Qcal)$ because it is finite.

\item This is also true for any matroid $M_{\Phi}(\Qcal)$, by a similar argument to that for (5).

\item If $I$ contains an even number of infinite circuits we may proceed as for (5). But if $I$ contains an odd number of infinite circuits, this will fail for some choices of $\Phi$. For example, for $i \in \{1, 2, 3\}$, let $w_i \colon \Nbb \to \{0, 1\}$ be defined by $w_i(n) = 1$ if $n$ is congruent to $i$ modulo 3 and $w(n) = 0$ otherwise. Let $o_i = o(0, w_i)$. Let $\Phi = \{[o_1], [o_2], [o_3]\}$, so that each $o_i$ is a circuit of $M_{\Phi}(\Qcal)$. It is clear that $o_1 \triangle o_2 \triangle o_3 = o(0, n \mapsto 1)$ includes no infinite circuit of $M_{\Phi}(\Qcal)$.

In order for this condition to be true of $M_{\Phi}$, it is necessary and sufficient for $\Phi$ to be closed under the ternary operation $$f: [o(1, v_1)], [o(1, v_2)], [o(1, v_3)] \mapsto [o(1, \chi_{\{n \in \Nbb | \{i | v_i(n) = 1\} \mbox{\small\ has 1 or 3 elements}\}})]$$

\item Like (7), this condition holds of $M_{\Phi}(\Qcal)$ precisely when $\Phi$ is closed under the ternary operation $f$.

\item Let $o = o(0, v)$ be some infinite circuit in $\bigcup \Phi$, and $s = o(0, w)$ some infinite circuit which is not in $\bigcup \Phi$. Then it is easy to check that $s$ is a base of $M_{\Phi}(\Qcal)$, and that whenever $v(i) \neq w(i)$ the fundamental circuit of $b^{v(i)}_i$ with respect to $s$ contains $a_1$. Since this happens infinitely often, the expression $\bigtriangleup_{e \in o \setminus s} o_e$ is not well defined (it is not a thin symmetric difference), so that this condition always fails for at least one base. 

However, there is also always at least one base $s$ for which we do have $o = \bigtriangleup_{e \in o \setminus s}o_e$ for any circuit $o$ of $M_{\Phi}(\Qcal)$, namely $\{a_1\} \cup \{b^0_i | i \in \Nbb\} \cup \{c^1_i | i \in \Nbb\}$. This choice of base works because it is already a base in $M_{\{\omega\}}(\Qcal)$.
\end{enumerate}

Note that the matroids constructed in \autoref{ray_case} are almost never tame. In fact, any intersection of a circuit in $\Phi$ with a cocircuit in $\Phi^*$ must be prolonged and so infinite. Thus the only cases in which $M_{\Phi}(\Rcal)$ might be tame are those with $\Phi$ empty or equal to the whole set of prolonged circuits.

\section{Proof of the main result}\label{sec:main_result}

We noted above that, if we have a matroid $N$ with a tree-decomposition $(T, R)$ of adhesion 2 then we often cannot reconstruct $N$ from the induced tree of matroids $\Tcal(N,T,R)$. In this section we will explain what additional information is needed to reconstruct $N$. 

In \autoref{sec:ray_case}, we showed that if $T$ is a ray then the extra information we need is the set $\Phi$ of $\simeq$-equivalence classes containing canonical precircuits of circuits of $N$. For more general tree-decompositions $(T, R)$, we will need to make such a specification for each end of $T$.

For each ray $Q\se T$, by \autoref{X1}, there is a minor $N_Q$ of $N$ with tree-decomposition $(Q,R_Q)$, with $R_Q$ chosen so that $\Tcal(N_Q, Q, R_Q)$ is obtained by restricting $\Tcal(N, T, R)$ to $Q$. Since $N_Q$ is a matroid, we get a set $\Phi(Q)$ of $\simeq$-equivalence classes containing canonical precircuits of circuits of $N_Q$. Note that the sets $\Phi(Q)$ for the rays $Q$ in some end $\omega$ are easily determined in terms of each other.

This suggests that the extra information we use should consist of a choice, for each end $\omega$, of a set $\Phi(Q_{\omega})$ of $\simeq$-equivalence classes for some chosen ray $Q_{\omega}$ belonging to $\omega$. This then determines a choice of $\Phi(Q)$ for each other ray $Q$ of $T$. The construction we have just outlined gives a choice of such a $\Phi$ for each tree-decomposition $(T,R)$ of a matroid $N$. We will denote this choice by $\Phi(N, T, R)$.

Given a tree $\Tcal = (T, M)$ of matroids and a specification of $\Phi(Q_\omega)$ for each end $\omega$ of $T$ as above, we say that a $\Tcal$-precircuit $(S_o, \hat o)$ is a $\Phi$-precircuit if and only if for each ray $Q$ of $S_o$ we have $(Q, \hat o \restric Q) \in \bigcup(\Phi(Q))$. A $\Phi$-circuit is a minimal nonempty underlying set of a $\Phi$-precircuit. Dually, a $\Tcal$-precocircuit $(S_b, \hat b)$ is a $\Phi^*$-precocircuit if and only if for each ray $Q$ of $S_b$ we have $(Q, \hat b \restric Q) \in \bigcup((\Phi(Q))^*)$, and we define $\Phi^*$-cocircuits dually to $\Phi$-circuits. If there is a matroid whose circuits are the  $\Phi$-circuits of $\Tcal$ and whose cocircuits are the $\Phi^*$-cocircuits of $\Tcal$ then we will call that matroid the {\em $\Phi$-matroid} for $\Tcal$, and denote it $M_{\Phi}(\Tcal)$.

The purpose of this section is to prove the following.

\begin{thm}\label{recon:arbi}
 Let $(T,R)$ be a tree-decomposition of a matroid $N$ of adhesion $2$.
Then $N$ is a $\Phi$-matroid for $\Tcal(N,T,R)$.
\end{thm}

Note that this is stronger than the claim that each circuit of $N$ arises from a $\Phi$-precircuit: we will show that the circuits of $N$ are precisely the $\Phi$-circuits.

\vspace{0.3 cm}

We will now begin the proof of \autoref{recon:arbi}. More precisely, we will now show that $N = M_{\Phi}(\Tcal(N, T, R))$, where $\Phi = \Phi(N, T, R)$. By definition, every $N$-circuit is the underlying set of some $\Phi$-precircuit
For $N$-cocircuits, we get the analogous fact.

\begin{lem}\label{cocir_underlying}
 Every $N$-cocircuit is an underlying set of a $\Phi^*$-precocircuit.
\end{lem}

\begin{proof}
Suppose not for a contradiction. That is, there are an $N$-cocircuit $b$ 
and an $N$-circuit $o$ such that the underlying set $S_b$ of the canonical
precocircuit $(S_b,\hat b)$ for $b$ and the underlying set $S_o$ for $\hat o$ have a common end $\omega$ such that they are $\simeq$-equivalent at $\omega$. Thus there is a ray $Q=q_1q_2\ldots$ converging to $\omega$ in the intersection of $S_b$ and $S_o$.

If the decomposition tree is a ray, we obtain a contradiction to \autoref{ray_case}.
The other case can be reduced to this case as follows:
We define $R^Q$ to be the following coarsening 
of the torsos $R$. We define $R^Q_{q_i}$ to be the union of all the $R_{v}$ such that in $T$ the vertices
$v$ and $q_i$ can be joined by a path that does not contain any other $q_j$.
Then $(Q,R^Q)$ is a tree-decomposition of $N$ of adhesion 2.
But the precircuit and precocircuit defined by $o$ and $b$ are also $\simeq$-equivalent at $\omega$ for the tree-decomposition $(Q,R^Q)$, contradicting \autoref{ray_case}.
\end{proof}

The key lemma is the following:

\begin{lem}\label{precir_give_scrawl}
 The underlying set $o$ of any $\Phi$-precircuit $(S_o,\hat o)$ is an $N$-scrawl.

Dually, the underlying set $b$ of any $\Phi$-precocircuit $(S_b,\hat b)$ is an $N$-coscrawl.
\end{lem}

In fact the theorem follows easily from this Lemma.

\begin{proof}[Proof that \autoref{precir_give_scrawl} implies \autoref{recon:arbi}.]
By \autoref{precir_give_scrawl}, every $\Phi$-circuit $o'$ is an $N$-scrawl and so includes an $N$-circuit $o$. By minimality, $o'$ must be $o$ since $o$  is the underlying set of the canonical precocircuit of $o$.
Conversely, every $N$-circuit $o$ is an underlying set of its canonical precircuit. Suppose for a contradiction that there is a $\Phi$-precircuit whose underlying set is properly included in $o$. Then this underlying set is an $N$-scrawl, so $o$ properly includes an $N$-circuit, which is impossible. So $o$ is a $\Phi$-circuit. Summing up, the $\Phi$-circuits are precisely the $N$-circuits.

The dual argument shows that the $N$-cocircuits are precisely the $\Phi$-cocircuits.
This completes the proof. 
\end{proof}

In order to prove the first statement of \autoref{precir_give_scrawl}, it suffices by \autoref{is_scrawl} to prove that $|o\cap b|\neq 1$ for any underlying set $o$ of a $\Phi$-precircuit $(S_o,\hat o)$ and any $N$-cocircuit $b$ with canonical
precocircuit  $(S_b,\hat b)$. 
Now suppose for a contradiction that there are such $o$, $(S_o,\hat o)$, $b$ and $(S_b,\hat b)$ with a single element $e$ in $o \cap b$.

\begin{comment}
 The rest of this section is devoted to the proof of \autoref{precir_give_scrawl}.
For the remainder of this section let us fix a $\Phi$-precircuit $(S_o,\hat o)$ with underlying set $o$, and an $N$-cocircuit $b$ with canonical
precircuit  $(S_b,\hat b)$. In order to show the first part of \autoref{precir_give_scrawl}, it suffices to prove that $|o\cap b|\neq 1$ for any such $o$ and $b$.
Now suppose for a contradiction that $|o\cap b|=1$.
\end{comment}

Let $N'$ be the matroid obtained from $N$ by applying \autoref{X1} to the subtree $S_o$ of $T$. Then $\Phi(N', S_o, R\restric_{S_o})$ is just $\Phi$ restricted to the set of ends of $S_o$ by the characterization of the $N'$-circuits in \autoref{X1}. 
Clearly, $(S_o,\hat o)$ is a $\Phi(N', S_o, R\restric_{S_o})$-precircuit.
By \autoref{X1}, the underlying set of $(S_b\cap S_o,\hat b\restric_{S_o})$ is an $N'$-cocircuit.
This construction shows that we may assume without loss of generality that $S_o=T$.

\begin{comment}
 TODO Sort out the construction in the next paragraph.
Use Lemma X1

Let $N'$ be the matroid obtained from $N$ by deleting all elements $e$ that are in matroids associated to
nodes of $T$ that are not in $S_o$. The tree-decomposition $(T,R)$ induces a tree-decomposition $(S_o,R')$ of $N'$ with $R'(v)$ being obtained from $R(v)$ by deleting all virtual elements $e(vw)$ such that $vw\not\in E(S_o)$. Let $\Phi'$ be defined from $N'$ and $(S_o, R')$ just as we defined $\Phi$ from $N$ and $(T, R)$. Note that $b$ restricted to the ground set of $N'$ is a coscrawl, so it includes some cocircuit $b'$ containing $e$. We will show that $(S_o, o')$ is a $\Phi'^*$-precircuit. For each end $\omega$ of $S_o$, there is some circuit $o$ of $N$ with For any $N'$-circuit $o_1$, But $b'$ 
This construction shows that we may assume without loss of generality that $S_o=T$.
\end{comment}

\begin{lem}\label{binary_tree}
The tree $S_b$ includes a subdivision of the binary tree $T_2$.
\end{lem}

Before we prove \autoref{binary_tree}, we need the following definition. Let $t_e$ be the node of $T$ containing $e$.
We say that a directed edge $tt'$ of $S_b$ \emph{points away from $e$} if $t_e$ is not in the subtree $T_{t\to t'}$.
A node of a tree is called a \emph{branching vertex} if it has degree at least 3.

\begin{proof}
The binary tree in $S_b$ will be built in two steps.
First we show that for every directed edge $tt'$ of $S_b$ pointing away from $e$,
the subtree $S_{t\to t'}$ includes a ray, then we show that it contains a branching vertex of $S_b$.
This will then show that $S_b$ includes a binary tree since $S_b$ includes at least one edge incident with $t_e$ as $\hat o(t_e)$ and $\hat b(t_e)$ have to intersect in an element different from $e$. Thus it remains to prove the following two sublemmas.

\begin{sublem}\label{ray_in_S}
 Let $tt'$ be a directed edge $tt'$ of $S_b$ pointing away from $e$.
Then the subtree $S_{t\to t'}$ includes a ray.
\end{sublem}

\begin{proof}
We construct the ray as follows. Let $t_1=t$ and $t_2=t'$, and suppose that $t_i$ is already defined. Then for $t_{i+1}$ we pick a neighbour of $t_i$ in $S_b$ such that $t_it_{i+1}$ points away from $e$. There is such a choice since $\hat o(t_i)$ and $\hat b(t_i)$ have to intersect in an element different from $e(t_{i-1}t_i)$.
Then $(t_i|i\in\Nbb)$ is the desired ray.
\end{proof}

\begin{sublem}\label{branching_in_S}
 Let $tt'$ be a directed edge $tt'$ of $S_b$ pointing away from $e$.
Then the subtree $S_{t\to t'}$ contains a branching vertex of $S_b$.
\end{sublem}

\begin{proof}
Suppose for a contradiction that $S_{t\to t'}$ does not contain a branching vertex.
Then by \autoref{ray_in_S} it is a ray: $S_{t\to t'}=q_1q_2\ldots$. 
By \autoref{X1}, there is a minor $N'$ of $N$ that has a tree-decomposition of adhesion $2$
with tree $S_{t\to t'}$ such that the matroid associated to $q_i$ is $R(q_i)$.
Then $\Phi(N', S_{t\to t'}, R\restric_{S_{t\to t'}})$ is just $\Phi$ restricted to 
the end of $S_{t\to t'}$. 
The restrictions of $(S_o,\hat o)$ and $(S_b,\hat b)$ to $S_{t\to t'}$ give a precircuit and a precocircuit for that new tree-decomposition which are not $\simeq$-equivalent at the end. This contradicts
\autoref{ray_case}, which completes the proof.
\end{proof}
\end{proof}

The strategy for the rest of this proof will be to improve $b$ in two steps, first to $b_1$, then to $b_2$, both satisfying the same conditions as $b$.
Let $(S_{b_1},\hat b_1)$ and $(S_{b_2},\hat b_2)$ be their canonical precocircuits. In the first step we force $S_{b_1}$ to have many vertices of degree 2. In the second step we force all but one of the vertices of $S_{b_2}$ to have degree 2 so that $S_{b_2}$ does not include a binary tree. By \autoref{binary_tree}, we then get the desired contradiction.

Before defining $b_1$, we construct a set $B\se S_b$ satisfying the following conditions.

\begin{enumerate}
 \item $t_e\in B$.
\item Every ray in $S_b$ meets $B$.
\item Every ray in $S_b$ meets some connected component of $S_b\sm B$ such that one of its nodes contains a (non-virtual) element of $b$.
\item Every vertex of $B$ is a branching vertex of $S_b$. 
\end{enumerate}

\begin{lem}\label{there_is_B}
 There exists $B\se S_b$ satisfying 1-4.
\end{lem}

Whilst the conditions 1 and 4 are easy to satisfy, the conditions 2 and 3 are in competition in the sense that putting more nodes into $B$ makes it more likely that condition 2 is true but  less likely that condition 3 is true.

\begin{proof}
We shall construct a sequence of nested rayless subtrees $T_n$ of $S_b$ that will exhaust the whole of $S_b$. The set $B$ will be a union of all the sets of leaves of the $T_n$ with $n$ odd (we will ensure that all such leaves lie outside $T_{n-1}$). 

To start with the construction, let $T_1=T_2=\{t_e\}$. 
Now suppose that $T_n$ is already constructed, where $n$ is even. 
Let $\partial T_n$ be the set of those nodes 
$x$ of $S_b$ that have a neighbour $t(x)$ in $T_n$ but are not in $T_n$, see 
\autoref{fig:construction_B}. 

 \begin{figure}
\begin{center}
 \includegraphics[width=5cm]{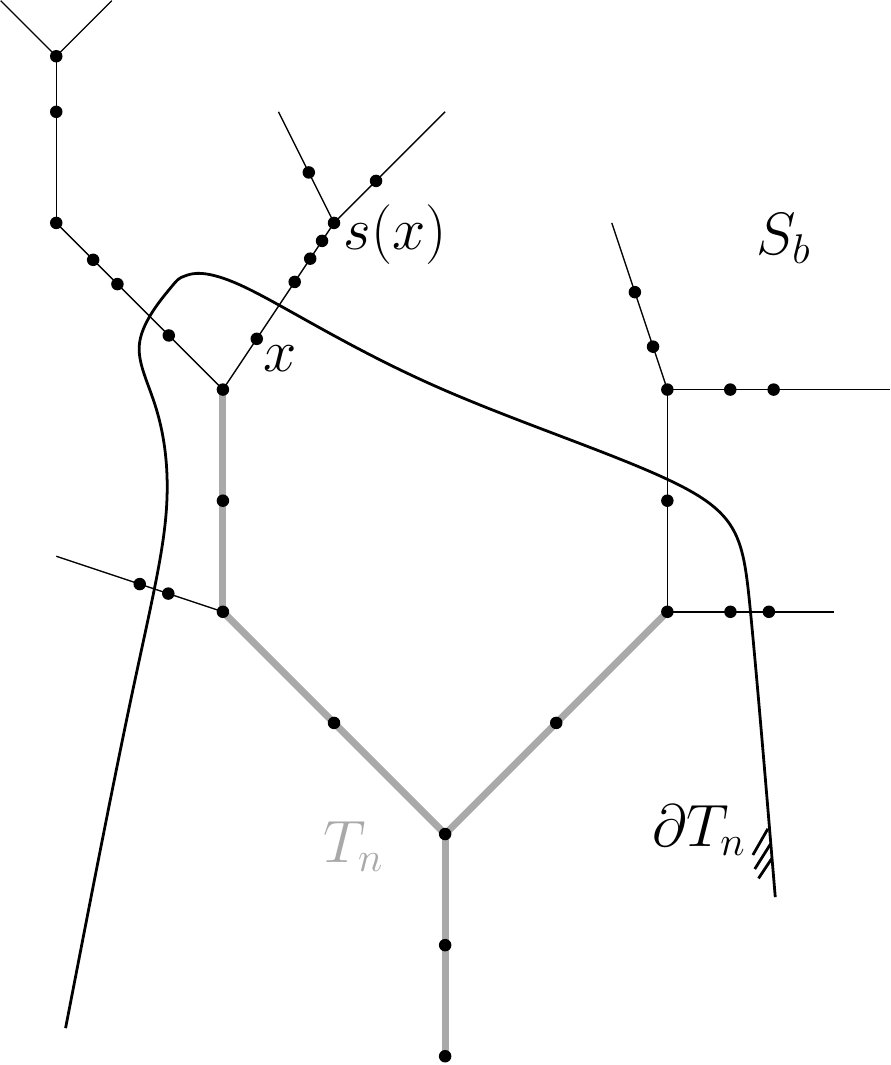}
\end{center} 
\caption{The construction of $T_{n+1}$ for $n$ even.}\label{fig:construction_B}
\end{figure}

By \autoref{branching_in_S}, for each $x\in \partial T_n$ the subtree ${(S_b)}_{t(x)\to x}$ contains a unique branching vertex $s(x)$ that is nearest to $x$. We obtain $T_{n+1}$ from $T_n$ 
by adding for every $x\in \partial T_n$ the unique $s(x)-T_n$-path. 
In particular, $\partial T_n\se T_{n+1}$.

Now assume that $T_n$ is already constructed with $n\geq 3$ odd. 
%Let $\partial T_n$ be the set of those nodes 
%$x$ of $S_b$ that have a neighbour $t(x)$ in $T_n$ but are not in $T_n$.
Since $(S_b,\hat b)$ is the canonical precocircuit for $b$, for each $x\in \partial T_n$ the subtree 
${(S_b)}_{t(x)\to x}$ includes a node $r(x)$ such that $b$ contains a non-virtual element of the matroid associated to $r(x)$, and by the choice of $\partial T_n$ this subtree does not meet $T_n$. 
We obtain $T_{n+1}$ from $T_n$ 
by adding for every $x\in \partial T_n$ the unique $r(x)-T_n$-path. 
This completes the definition of the $T_n$.

As revealed above, $B$ is the union of all the sets of leaves of the $T_n$ with $n$ odd. 
Clearly $B$ satisfies conditions 1 and 4.

Next we show that $B$ satisfies condition 3.
For each $n$, the set $T_{2n}\sm T_{2n-1}$ is a forest with the property that each of its components includes one vertex $r(x)$.
Hence $B$ satisfies condition 3, since each ray meets $T_{2n}\sm T_{2n-1}$ for some sufficiently large $n$.

It remains to show that $B$ satisfies condition 2. 
By construction, the $T_n$ form a sequence of nested rayless subtrees of $S_b$ that exhaust the whole of $S_b$. 
Now let $Q$ be a ray of $S_b$. Thus there is a large number $n$ such that 
$S_b$ has a node in $T_{2n}$. Let $x$ be the last node of $Q$ in $T_{2n+1}$.
By construction of $T_{2n+1}$, the node $x$ cannot have degree 2, so it is branching.
Since $x$ is not in $T_{2n}$, it must be a leaf of $T_{2n+1}$, and thus is in $B$,
which completes the proof.  
 
\begin{comment}
 We shall construct $B$ as a nested union of finite sets $B_n$ where $B_1=\{e\}$.
Now suppose that $B_{n}$ is already constructed. Let $\partial B_n$ be the set of all
 neighbours $\bar t$ in $S_b$ of some $t\in B_n$ such that ${(S_b)}_{t\to \bar t}$ avoids $B_n$. For each $\bar t\in \partial
 B_n$ we pick some $t'\in S_b$ such that $t'\in {(S_b)}_{t\to \bar t}$ and $b$ uses
 a non-virtual element at the matroid associated to $t'$. Note that such a choice is possible since
 $(S_b,\hat b)$ is the canonical precocircuit for $b$. Let $T'$ be the set of such $t'$. Let
 $\partial T'$ be the set of all neighbours $t$ of some $t'\in T'$ such that $t't$ points away from $e$.
For each such $x\in \partial T'$ we pick a node $t_x$ such that $t_x$ is further away from $e$ as $x$ is and $t_x$ is a branching vertex; such a choice 
is possible by \autoref{branching_in_S}.

Now let $B_{n+1}$ be the union of $B_n$ with all such $t_x$.
Let $B=\bigcup_{n\in\Nbb} B_n$. Then clearly $B$ satisfies 1-4 by construction.
\end{comment}

\end{proof}

Now we are in a position to define $b_1$. Let $X=\left[\bigcup_{v\in B} E(v)\right]\sm o$.
Applying $(O3)^*$ to $b,X$ and $e$ we get a cocircuit $b_1$ with $e\in b_1\se b\cup X$ with
$b_1\sm X$ minimal subject to these conditions. Let $(S_{b_1},\hat b_1)$ be the canonical cocircuit for $b_1$.
In the next few lemmas we collect some properties of $b_1$ and $S_{b_1}$.

\begin{lem}\label{b_1_agree}
For any $v\in S_{b_1}\sm B$, we have $\hat b_1(v)=\hat b(v)$.
\end{lem}

\begin{proof}
 It suffices to show that the possible elements for $\hat b_1(v)$ are all in $\hat b(v)$.
Since $S_{b_1}\se S_b$, all possible virtual elements are in $\hat b(v)$.
The possible non-virtual elements are in $\hat b(v)$ since $v\not\in B$.
\end{proof}

\begin{lem}\label{b_1_deg2}
Each $v\in S_{b_1}\cap B$ has degree 2 in $S_{b_1}$.
\end{lem}

\begin{figure}
\begin{center}
 \includegraphics[width=12cm]{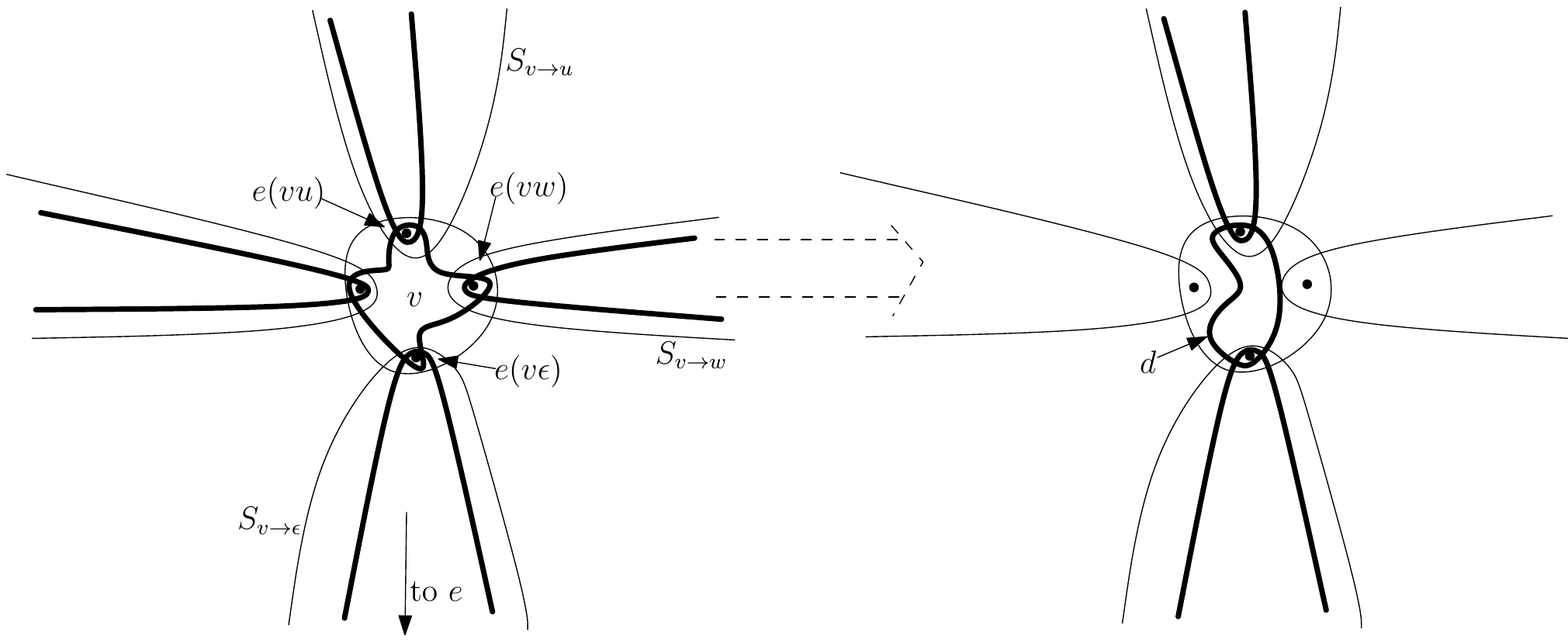}
\end{center} 
\caption{The construction of $(S_a, \hat a)$}\label{fig:Sa}
\end{figure}

\begin{proof}
Suppose for a contradiction, that there is some $v\in S_{b_1}\cap X$ such that its degree is not $2$. Since $S_{b_1}$ contains $t_e$, and $\hat b_1(t_e)$ cannot meet $\hat o(t_e)$ only in  $e$, the tree $S_{b_1}$ has at least one edge. As $S_{b_1}$ is also connected, the degree of $v$ in $S_{b_1}$ is at least $1$. Thus the degree is at least 2 because otherwise 
$\hat b_1(v)$ and $\hat o(v)$ would just intersect in a single element. 

Now suppose for a contradiction that the degree of $v$ in $S_{b_1}$ is at least $3$.
So there are three neighbours $u$, $w$ and $\epsilon$ of $v$ such that 
$e(vu),e(vw),e(v\epsilon)\in \hat b_1(v)$, and
the directed edges $vu$ and $vw$ point away from $e$, whilst $v\epsilon$ points towards $e$.

Let $d$ be an $M(v)$-cocircuit that meets $\hat o(v)$ in precisely $e(vu)$ and $e(v\epsilon)$; such exists by \autoref{o_cap_b}.
Let $S_a={(S_{b_1})}_{v\to u}\cup {(S_{b_1})}_{v\to \epsilon}+v$. 
We obtain $\hat a$ from the restriction of $\hat b_1$ to $S_a-v$ by extending it to $S_a$ via 
$\hat a(v)=d$. This construction is illustrated in \autoref{fig:Sa}.

Applying \autoref{gluing2} twice, once to the separation corresponding to $vu$ and once
to the separation corresponding to $v\epsilon$, we get that $N=N(w)\oplus_2 M(v) \oplus_2 N(\epsilon)$ where the ground set of $N(w)$ is the set of those elements that are in $R(t)$ with $t\in T_{v \to w}$. Thus the precocircuit $(S_a,\hat a)$ is an $N$-cocircuit.

By \autoref{ray_in_S}, the subtree ${(S_{b_1})}_{v\to w}$ includes a ray $Q$. By condition 3, this ray meets a component $D$ of $S_b\sm B$ that includes a non-virtual element.
By \autoref{b_1_agree} we have $D\se {(S_{b_1})}_{v\to w}$, so there is a non-virtual element not in $X$ used by $b_1$ but not by this new $N$-cocircuit. Hence this new cocircuit violates the minimality of $b_1$,
which gives the desired contradiction. Thus $v$ has degree  precisely 2 in $S_{b_1}$.

\end{proof}

\begin{lem}\label{b_1_real}
For any $v\in S_{b_1}\cap B$, the cocircuit $\hat b_1(v)$ includes at least one non-virtual element.
\end{lem}

\begin{proof}
By the property 4 of $B$, the cocircuit $\hat b(v)$ includes at least 3 virtual elements. Hence 
$\hat b_1(v)\neq \hat b(v)$ and $\hat b_1(v)$ has to contain an element not in $\hat b(v)$
which must be real.
\end{proof}

Having collected some properties of $b_1$ and $S_{b_1}$, we next define $b_2$ in a similar way to $b_1$.
Let $B'=V(S_{b_1})\sm B$, and let $X'=\left[\bigcup_{v\in B'} E(v)\right]\sm o$.
Applying $(O3)^*$ to $b_1,X'$ and $e$ we get a cocircuit $b_2$ with $e\in b_2\se b_1\cup X'$ with
$b_2\sm X'$ minimal subject to these conditions. Let $(S_{b_2},\hat b_2)$ be the canonical cocircuit for $b_2$.
Just as before, for every $v\in (S_{b_2}\sm B')-t_e$, we have $\hat b_2(v)=\hat b_1(v)$.

\begin{lem}\label{all_deg_2}
 Each $v\in (S_{b_2}\cap B')-t_e$ has degree 2 in $S_{b_2}$.
\end{lem}

\begin{proof}
This is proved in the same way as \autoref{b_1_deg2} with $B'$, $X'$, $b_2$ and $S_{b_2}$ in place of $B$, $X$, $b_1$ and $S_{b_1}$. Indeed, $B'$ satisfies condition 2 since $B$ satisfies condition 3, and it satisfies condition 3 by \autoref{b_1_real}.
Conditions 1 and 4 were not used in the proof of \autoref{b_1_deg2}.
\end{proof}

Hence every vertex except for $t_e$ has degree 2 in $S_{b_2}$. Hence $S_{b_2}$ does not include a subdivision of the binary tree. Since $b_2$ is a legal choice for $b$, we get a contradiction to \autoref{binary_tree}. This completes the proof of the first part of \autoref{precir_give_scrawl}. The proof of the second part is the dual argument. (At first glace this might look strange since the definition of $\Phi$ is asymmetric. However \autoref{cocir_underlying} deals with this asymmetry and the last proof does not rely on the choice we made when defining $\Phi$.) Hence this completes the proof of \autoref{recon:arbi}.

If $N$ is tame then we need very little local information at the end. In fact, by the comments at the end of the last section, for any end $\omega$ one of $\Phi(Q_{\omega})$ or $(\Phi(Q_{\omega}))^*$ must be empty, and \autoref{tamemain} follows easily.

We need equally little information at the ends if the tree of matroids is planar.

\begin{proof}[Proof that \autoref{cor:planar} follows from \autoref{recon:arbi}]
It is enough to prove \autoref{cor:planar} for a nice ray $\Rcal$ of matroids 
with all local matroids being 2-connected and planar with both virtual elements on the outer face.
Let $o$ be the element of $\Ccal_\infty$ which at each local matroid takes the outer face.
It now suffices to show that any prolonged circuit $o'\in \Ccal_\infty$ of $\Rcal$ is $\simeq$-equivalent to $o$.

Let $b \in \Dcal_\infty$ that is disjoint from $o'$, which exists as in each local matroid there is a cocircuit meeting the local circuit of $o'$ precisely in the virtual elements. 
Then $o\sim b\sim o'$, yielding that $o$ and $o'$ are $\simeq$-equivalent.
\end{proof}

\section{Example showing that all our information at the ends is necessary}\label{example_binary_tree}

Let $T_2$ be the infinite binary tree. 
While the vertices of $T_2$ are the finite 0-1-sequences, the ends of $T_2$ are the infinite ones and $\omega\restric_{n}$ consists of the first n digits of $\omega$. 

We obtain $W$ from $T_2$ by replacing each vertex by two non-adjacent vertices and joining two new vertices by an edge if the vertices they come from are adjacent. 
Then we join the two vertices replacing the root, see \autoref{fig:W}.
Formally $V(G)=T_2\times \{1,2\}$, where $(v,x)$ and $(w,y)$ are adjacent if $v$ and $w$ are or $v=w=\text{root}$.
\begin{figure}[htb]
\begin{center}
 \includegraphics[width=8cm]{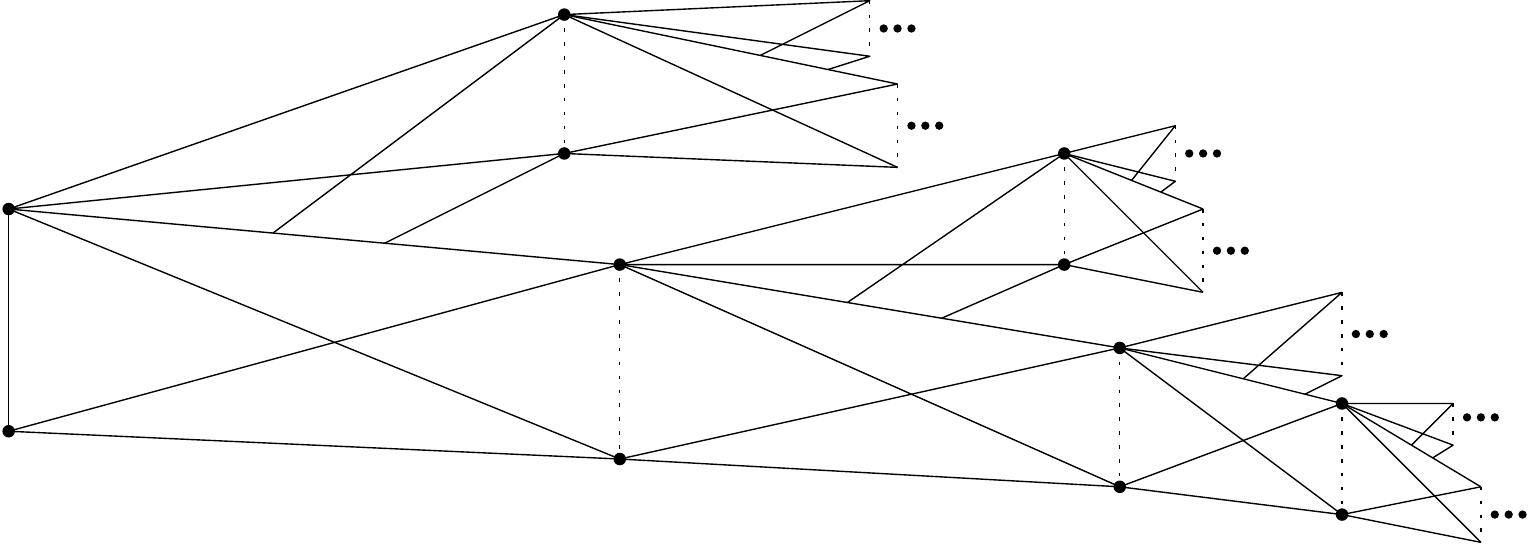}
\end{center} 
\caption{The graph $W$.} \label{fig:W}
\end{figure}

The canonical tree-decomposition of $W$ into 3-connected minors is the following:
\begin{itemize}
\item The decomposition-tree is $\dot T_2$, the binary tree with each edge subdivided;
\item The torsos of the subdivided edges are $K_4$s, where the virtual elements are non-adjacent;
\item The torsos of the other vertices consist of 3 elements in parallel, which are all virtual elements.
\end{itemize}

The \emph{bottom} virtual element of a torso of this tree-decomposition is the one whose corresponding element of $\dot T_2$ is nearer to the root of $\dot T_2$ (the root of $\dot T_2$ is the vertex that is also the root of $T_2$). 
Let $\Tcal$ be the tree of matroids with underlying tree  $\dot T_2$, which has $M(K_4)$s at all subdivision vertices and $U_{1,3}$ at all other vertices. Once more the virtual elements of the $M(K_4)$s are not in a common triad. 

By \autoref{recon:arbi}, every matroid that has $\Tcal$ as its canonical tree of matroids is a $\Phi$-matroid.
Next, we shall derive a more explicit description of the $\Phi$-classes at each end.
There are natural bijections between the ends of $\dot T_2$, $T_2$ and $W$, and we shall suppress them in our notation.

By $\Fcal$ we denote the set of circuits of $M_{TC}(W)$: these are the edge sets of finite cycles of $W$, double rays of $W$ with both tails converging to the same end, and pairs of vertex-disjoint double rays which both go to the same ends.

Let $o\in \Fcal$ converge to the end $\omega$. Then $o$ \emph{induces} a $\{+,-\}$-sequence of infinite length towards $\omega$ as follows:
the $n$-th digit is + if both $(\omega\restric_{n},1) (\omega\restric_{n+1},1)$ and $(\omega\restric_{n},2) (\omega\restric_{n+1},2)$ are elements of $o$, otherwise it is -. Note that it is eventually true that if the $n$-th digit is -, then 
$(\omega\restric_{n},1) (\omega\restric_{n+1},2)$ and $(\omega\restric_{n},2) (\omega\restric_{n+1},1)$ are elements of $o$. 

It is not difficult to check that two circuits are in the same $\Phi$-class at $\omega$ if and only if their induced $\{+,-\}$-sequences agree. 
Thus every $\Phi$-matroid of $\Tcal$ is uniquely determined by a choice  for each end of $T_2$ of a subset of $\{+,-\}^{\Nbb}$ that is closed under finite changes.
For any such choice $\tau$, an element $o$ of $\Fcal$ is $\tau$-legal if $o$ induces only $\{+,-\}$-sequences in $\tau$ at all ends to which $o$ converges. 

In this section we prove that for any such choice $\tau$ we do get a matroid.

\begin{thm}\label{many_matroids}
Let $\tau$ be a choice  for each end of $T_2$ of a subset of $\{+,-\}^{\Nbb}$ that is closed under finite changes.
Then the set $\Ccal$ of  $\tau$-legal elements of $\Fcal$ is the set of circuits of a matroid. 
\end{thm}

First we need some preparation.
The edge set of $K_4$ admits a unique 3-partition $E(K_4)=X_1\dot \cup X_2\dot \cup X_3$ into three pairs of non-adjacent edges. 
We shall need the following property of $K_4$:

\begin{rem}\label{parallel_edges}
Let $Y_1\dot \cup Y_2$ be a partition of $X_1\dot \cup X_2$, and let $e\in X_3$. Then $Y_1$ spans $e$, or $Y_2$ cospans $e$ or 
$\{X_1,X_2\}=\{Y_1,Y_2\}$.
\qed  
\end{rem}

\begin{proof}[Proof of \autoref{many_matroids}.]
Let $\Dcal$ be the set of those bonds $b$ of $W$ that have infinite intersection with every $\tau$-legal element
of $\Ccal$ converging to some end to which $b$ converges.
We shall apply \autoref{hybrid} in order to show that $\Ccal$ is the set of circuits of a matroid. Thus it remains to show that $\Ccal$ and $\Dcal$ satisfy (O2) and (CM). 
Indeed, (O1) is clear since if a topological cycle and a bond of a locally finite graph intersect just in a single edge, then there is an end to which they both converge, see for example \cite[Lemma 2.6]{BC:ubiquity}.\footnote{In this section all topological cycles will have that shape. However in general graphs topological cycles may be more complicated, see \cite{RD:HB:graphmatroids} for a definition of the topological-cycle matroid of a locally finite graph.} 

First we check (O2) for some partition $E=P\dot \cup Q\dot\cup \{e\}$, where $e$ is the edge joining the two vertices replacing the root of $T_2$. We may assume that there is no finite $o\in \Ccal$ with $e\in o\se P+e$ and that there is no finite 
$b\in \Dcal$ with $e\in b\se Q+e$. 

We say that a subdivided edge of $\dot T_2$ is \emph{blocking} if 
 the bottom virtual element of its torso is cospanned by the elements of that torso that are in $Q$.
In order to study the torsos that are isomorphic to $K_4$, we shall use the 3-partition of $E(K_4)$ discussed above, where we follow the convention that $X_3$ consists of the virtual elements. 
We say that a subdivided edge $u$ of $\dot T_2$ is \emph{undecided} if 
 $\{P\cap E(u),Q\cap E(u)\}=\{X_1,X_2\}$.

Let $U$ be the set of all subdivided edges of $\dot T_2$ that are not separated by a blocking vertex from the root. 
Any $t\in U$ is undecided. To see this just consider the first subdivided edge above the root and below $t$ that is not undecided and apply \autoref{parallel_edges}. 

Let $\bar U$ be the set of those blocking vertices that only have vertices in $U$ below them.
For each $r\in \bar U$ we pick a bond $b_r$ witnessing that $Q\cap E(r)$ cospans the bottom element of the torso at $r$. 

Let $b'$ be the set of all edges of $Q$ that are in torsos for nodes in $U$ together with the non-virtual elements of all the bonds $b_r$ with $r\in \bar U$.
By construction $b=b'+e$ is a bond of $W$ containing $e$.
We may assume that some $\tau$-legal element $o$
of $\Ccal$ has finite intersection with $b$ and converges to some end $\omega$ to which $b$ converges.
Relying on the particular description of $\Fcal$, we can change $o$ so that we may assume without loss of generality that $o$ is a double ray and has only the end $\omega$ in its closure\footnote{We say that an end $\omega$ of a graph $G$ is in the \emph{closure of} an edge set $F$ if $F$ cannot be separated from $\omega$ by removing finitely many vertices from $G$.}. 
We get $o'$ from $o$ by changing $o$ to $E(u)\cap P$ at the finitely many torsos $u$ at which $o$ intersects $b$.  
Since all these torsos are undecided, $o'$ is a double ray and it is $\tau$-legal as $\tau$ is closed under finite changes. 
Thus $o'$ witnesses (O2) in this case.

Having proved (O2) in the case where $e$ is the edge joining the two vertices replacing the root of $T_2$, it remains to consider the case where is $e$ is an edge of some torso $u$. Let $S$ and $S'$ be the two components of $\dot T_2$ $- u$. The above argument yields (O2) at the virtual element of $\Tcal$ restricted to $S$. 
This and the corresponding fact for $S'$ imply (O2) in this case. We leave the details to the reader. 

Having proved (O2), it remains to prove (CM). For that let $I\se X\se E(W)$ be given. 
Let $Z$ be the set of those subdivided edges $z$ of $\dot T_2$ such that $X$ contains no non-virtual element of $E(z)$.
Let $u$ be a subdivided edge of $\dot T_2$ not in $Z$.
Then we may assume that $I$ contains at least one edge in $E(u)$ since the set of the 4 non-virtual elements of $E(u)$ meets no element of $\Fcal$ just once. 
Let $L$ consist of one edge of $I$ for each such torso.

Let $W'=W- X$. Let $D$ be a connected component of $W'$. We abbreviate $I'=I\cap E(D)$, $L'=L\cap E(D)$ and $X'=X\cap X'$, and our intermediate aim is to construct a set witnessing (CM) for $I'$ and $X'$. 

Let $W''=D'/J'$.
Let $K$ be the torso-matroid obtained from $M(K_3)$
by adding the two virtual elements in parallel to two distinct elements. Note that $W''$ has a tree-decomposition along a subgraph of $\dot T_2$
such that all torsos at subdivided edges are $K$, the other torsos are  $U_{0,1}$ at leaves and else $U_{1,3}$.
By doing Whitney-flips\footnote{The \emph{Whitney-flip} of a graph $G$ which is a 2-sum of two graphs $H_1$ and $H_2$ is the other graph that is also a 2-sum of these graph buts with the endvertices of the gluing element identified the other way round.} at virtual elements if necessary, we may assume that each vertex of $W''$ has finite degree. 
Let $\Ccal''$ consist of those topological circuits $o$ of $W''$ such that there is a subset $j$ of $J'$ such that $o\cup j\in \Ccal$. 

Note that if two elements of $\Ccal''$ have the same end in their closure, then their rays to that end eventually agree. 
It is easy to check for a topological circuit $o$ of $W''$ that if all its ends are also in the closure of (possibly two different) elements of $\Ccal''$, then
$o$ is in $\Ccal''$.
Let $\Psi$ be the set of ends in the closure of elements of  $\Ccal''$. By the above argument,  $\Ccal''$ is the set of $\Psi$-circuits. 
By \cite[Lemma 6.11]{BC:psi_matroids}, $\Ccal''$ is the set of circuits of a matroid (Actually the matroid here is a contraction-minor of the matroid in that theorem).
So let $J'$ witness (CM) for $I'\sm L'$ and $X'\sm L'$. 
Then $J'\cup L'$ witnesses (CM) for $I'$ and $X'$.
Furthermore it is easy to see that the union of all these witnesses for the different components $D$ of $W'$ witnesses (CM) for $I$ and $X$.

\end{proof}

\bibliographystyle{plain}
\bibliography{literatur}
\end{document}

%% file: minimalJ.tex
\usepackage{epsfig}
\usepackage{graphicx}
\usepackage{amsbsy}
\usepackage{amsmath}
\usepackage{amsfonts}
\usepackage{amssymb}
\usepackage{textcomp}
\usepackage{hyperref}
\usepackage{aliascnt}

\newcommand{\mcm}[3]{\newcommand{#1}[#2]{{\ensuremath{#3}}}} 
\newcommand{\sm}{\setminus}
\newcommand{\se}{\subseteq}
\newcommand{\ct}{^\complement}

\mcm{\tuple}{1}{\langle #1 \rangle}
\mcm{\name}{1}{\ulcorner #1 \urcorner}
\mcm{\Nbb}{0}{\mathbb{N}}
\mcm{\Zbb}{0}{\mathbb{Z}}
\mcm{\Qbb}{0}{\mathbb{Q}}
\mcm{\Rbb}{0}{\mathbb{R}}
\mcm{\Cbb}{0}{\mathbb{C}}
\mcm{\Fbb}{0}{\mathbb{F}}
\mcm{\Bcal}{0}{\cal B}
\mcm{\Ccal}{0}{\cal C}
\mcm{\Dcal}{0}{\cal D}
\mcm{\Ecal}{0}{\cal E}
\mcm{\Fcal}{0}{\cal F}
\mcm{\Gcal}{0}{\cal G}
\mcm{\Hcal}{0}{\cal H}
\mcm{\Ical}{0}{\cal I}
\mcm{\Lcal}{0}{\cal L}
\mcm{\Mcal}{0}{\cal M}
\mcm{\Ncal}{0}{\cal N}
\mcm{\Qcal}{0}{\cal Q}
\mcm{\Pcal}{0}{{\cal P}}
\mcm{\Scal}{0}{{\cal S}}
\mcm{\Tcal}{0}{{\cal T}}
\mcm{\Ucal}{0}{{\cal U}}
\mcm{\Vcal}{0}{{\cal V}}
\mcm{\Wcal}{0}{{\cal W}}
\mcm{\Xcal}{0}{{\cal X}}
\mcm{\Ycal}{0}{{\cal Y}}
\mcm{\Zcal}{0}{{\cal Z}}
\mcm{\Mfrak}{0}{\mathfrak M}

\mcm{\restric}{0}{\upharpoonright}
\mcm{\upset}{0}{\uparrow}
\mcm{\downset}{0}{\downarrow}
\mcm{\onto}{0}{\twoheadrightarrow}
\mcm{\smallNbb}{0}{{\small \mathbb{N}}}
\DeclareMathOperator{\preop}{op}
\mcm{\op}{0}{^{\preop}}

{\begin{array}{c}
\setlength{\unitlength}{1em}}%
{\end{array}}

\usepackage{amsthm}

\newcommand{\theoremize}[2]{\newaliascnt{#1}{thm} \newtheorem{#1}[#1]{#2} \aliascntresetthe{#1}}

\theoremstyle{plain}
\newtheorem{thm}{Theorem}[section]
\theoremize{lem}{Lemma}
\theoremize{sublem}{Sublemma}
\theoremize{claim}{Claim}
\theoremize{rem}{Remark}
\theoremize{prop}{Proposition}
\theoremize{cor}{Corollary}
\theoremize{que}{Question}
\theoremize{oque}{Open Question}
\theoremize{con}{Conjecture}

\theoremstyle{definition}
\theoremize{dfn}{Definition}
\theoremize{eg}{Example}
\theoremize{exercise}{Exercise}
\theoremstyle{plain}